\documentclass[12pt]{amsart}

\usepackage{amssymb}
\usepackage{amsmath}
\usepackage{amsbsy}
\usepackage{amscd}
\usepackage[mathscr]{eucal}
\usepackage{color}
\usepackage{amsrefs}
\usepackage[all,cmtip]{xy}
\usepackage{tikz-cd}
\usepackage{stmaryrd}
\advance\textheight by \topskip
\makeatletter
\def\cal{\mathcal}
\def\Bbb{\mathbb}

\newenvironment{pf*}[1]{\proof[#1]}{\endproof}
\renewcommand{\thesubsection}{\thesection(\@roman\c@subsection)}
\makeatother
\newtheorem{Theorem}[equation]{Theorem}
\newtheorem{Corollary}[equation]{Corollary}
\newtheorem{Lemma}[equation]{Lemma}
\newtheorem{Proposition}[equation]{Proposition}

\theoremstyle{definition}
\newtheorem{Definition}[equation]{Definition}
\newtheorem{Example}[equation]{Example}

\newtheorem*{Addendum}{Addendum}

\newtheorem*{acknowledgement}{Acknowledgement}

\makeatletter
\renewcommand\section{\@startsection{section}{1}%
  {\z@}{.7\linespacing\@plus\linespacing}{.5\linespacing}%
  {\reset@font\normalfont\bfseries\centering}}
\makeatother

\theoremstyle{remark}
\newtheorem{Remark}[equation]{Remark}

\numberwithin{equation}{section}
\numberwithin{figure}{section}

\newcommand{\thmref}[1]{Theorem~\ref{#1}}

\newcommand{\lemref}[1]{Lemma~\ref{#1}}
\newcommand{\propref}[1]{Proposition~\ref{#1}}
\newcommand{\remref}[1]{Remark~\ref{#1}}
\newcommand{\corref}[1]{Corollary~\ref{#1}}
\newcommand{\subsecref}[1]{\S\ref{#1}}

\newcommand{\Romnum}[1]{\expandafter\uppercase\expandafter{\romannumeral #1}} 
\newcommand{\C}{{\Bbb C}}
\newcommand{\Z}{{\Bbb Z}}
\newcommand{\Q}{{\Bbb Q}}
\newcommand{\R}{{\Bbb R}}
\newcommand{\HH}{{\Bbb H}}
\newcommand{\CP}{\operatorname{\C P}}
\newcommand{\RP}{\operatorname{\R P}}
\newcommand{\SO}{\operatorname{\rm SO}}
\newcommand{\U}{\operatorname{\rm U}}
\newcommand{\Pin}{\operatorname{\rm Pin}}
\newcommand{\Spin}{\operatorname{\rm Spin}}
\newcommand{\Spinc}{\Spin^{c}}

\newcommand{\SP}{\operatorname{\rm Sp}}
\newcommand{\End}{\operatorname{End}}
\newcommand{\Map}{\operatorname{Map}}

\newcommand{\Hom}{\operatorname{Hom}}

\newcommand{\Ima}{\operatorname{Im}}

\newcommand{\rank}{\operatorname{rank}}
\newcommand{\sign}{\operatorname{sign}}
\newcommand{\supp}{\operatorname{supp}}%

\newcommand{\ind}{\mathop{\text{\rm ind}}\nolimits}

\newcommand{\colim}{\operatorname{colim}}
\newcommand{\M}{{\cal M}} % the moduli space
\newcommand{\G}{{\cal G}} % the gauge transformation group

\newcommand{\id}{\operatorname{id}}

\newcommand{\deux}{\{\pm 1\}}
\newcommand{\Diff}{\operatorname{Diff}}
\newcommand{\Homeo}{\operatorname{Homeo}}
\newcommand{\Aut}{\operatorname{Aut}}
\newcommand{\fs}{\mathfrak{s}}

\begin{document}
\title[Rigidity of the Seiberg-Witten invariants for spin families]{Rigidity of the mod 2 families Seiberg-Witten invariants and topology of families of spin 4-manifolds}

\author{Tsuyoshi Kato}
\address{Department of Mathematics, Kyoto University, Japan}
\email{tkato@math.kyoto-u.ac.jp}
\author{Hokuto Konno}
\address{Graduate School of Mathematical Sciences, the University of Tokyo, 3-8-1 Komaba, Meguro, Tokyo 153-8914, Japan \\and\\
RIKEN iTHEMS, Wako, Saitama 351-0198, Japan}
\email{konno@ms.u-tokyo.ac.jp}
\author{Nobuhiro Nakamura}
\address{Department of Mathematics, Osaka Medical College, 2-7 Daigaku-machi, Takatsuki City, Osaka, 569-8686, Japan}
\email{mat002@osaka-med.ac.jp}

\begin{abstract}
We show a rigidity theorem for the Seiberg--Witten invariants mod 2 for families of spin 4-manifolds.
A mechanism of this rigidity theorem also gives a family version of 10/8-type inequality.
As an application, we prove the existence of non-smoothable topological families of 4-manifolds whose fiber, base space, and total space are smoothable as manifolds.
These non-smoothable topological families provide new examples of $4$-manifolds $M$ for which the inclusion maps $\Diff(M) \hookrightarrow \Homeo(M)$ are not weak homotopy equivalences.
We shall also give a new series of non-smoothable topological actions on some spin $4$-manifolds.
\end{abstract}

\maketitle

\tableofcontents

\section{Introduction}

The Seiberg--Witten invariant is an integer-valued differential topological invariant of a $\Spinc$ $4$-manifold, which reflects the smooth structure for various examples of $4$-manifolds.
Nevertheless, if the $\Spinc$-structure is induced from a spin structure, one may expect
 a sort of ``rigidity theorem'' for the Seiberg--Witten invariant mod $2$.
Namely, the value  of the Seiberg--Witten invariant mod $2$ may depend only on
some underlying topological structure of the smooth manifold, such as on  its homotopy type.
Such kind of results have been obtained by Morgan--Szab\'{o}~\cite{MS}, Ruberman--Strle~\cite{MR1809302}, Bauer~\cite{B}
and Li~\cite{MR2264722, MR1868921}.

Not only for a single $4$-manifold, various authors developed gauge theory for families of $4$-manifolds with many interesting applications, such as \cite{MR1671187,MR1734421,MR1874146,MR1868921,MR2652709,Baraglia,MR2644908}.
In particular, for a smooth family of $4$-manifolds, the families Seiberg--Witten invariant has been defined as an $\Z$- or $\Z/2$-valued invariant.

In this paper, we study a family version of rigidity results on the $\Z/2$-valued Seiberg--Witten invariant.
Namely, for a given family of spin $4$-manifolds with some topological conditions, we consider the $\Z/2$-valued families Seiberg--Witten invariant, and verify 
 that it depends only on weaker information than what is a priori expected.
Roughly speaking, we verify that the $\Z/2$-valued families Seiberg--Witten invariant is determined by the linearization of a family of Seiberg--Witten equations.
A mechanism of this rigidity theorem also gives a family version of Furuta's 10/8-inequality \cite{MR1839478} in a suitable situation.

This family version of 10/8-type inequality gives us the following topological applications:
we prove the existence of a {\it non-smoothable family} of $4$-manifolds whose fiber, base space, and the  total space are smoothable as manifolds.
To our knowledge, this interesting topological phenomenon has not been discussed so far.
This non-smoothability result gives a new approach to detect homotopical difference between the diffeomorphism and homeomorphism groups of $4$-manifolds.
For example, let $M$ be a smooth $4$-manifold which is homeomorphic to $K3\# n S^{2} \times S^{2}$ with $0\leq n \leq 3$. 
Then it follows from our non-smoothability result that the inclusion map from the diffeomorphism group to the homeomorphism group
\[
\Diff(M) \hookrightarrow \Homeo(M)
\]
is not a weak homotopy equivalence.
This result has not been known even when $M$ is diffeomorphic to $K3\# n S^{2} \times S^{2}$ with $n>0$.
As another application, we shall also detect a new series of non-smoothable topological actions on some spin $4$-manifolds using the family version of the 10/8-inequality.

Let us summarize the statements of our main theorems and their applications.
Henceforth, all manifolds are assumed to be connected.
Let $B$ be a closed smooth manifold,
$M$  a closed smooth $4$-manifold equipped with a spin structure $\fs$
and $M \to X \to B$ be a fiber bundle whose structure group is $\Diff^{+}(M)$, the group of diffeomorphisms preserving the orientation.
Assume that $X$ admits a fiberwise spin structure $\fs_{X}$ whose fiber coincides with the given spin structure on $M$.
 We call  it a {\it global spin structure} modelled on $\fs$
(See Subsection~\ref{subsection: Families Seiberg--Witten invariants}).
In this situation, we have two real bundles over $B$:
$
H^{+} \to B $ and 
$\ind{D}$,
where the fiber of $H^+$ is 
$H^{+}(M)$ which  is 
 a maximal-dimensional positive definite subspace of $H^{2}(M;\R)$ with respect to the intersection form, 
and $\ind{D}$ is 
 the virtual Dirac index bundle associated to $X \to B$.
 Note that the Dirac operator $D$ is $\Pin(2)$-equivariant since $D$ is $\HH$-linear.
 We define the $\Pin(2)$-action on $H^+$ via the surjective homomorphism $\Pin(2)\to \Pin(2)/S^1 =\deux$ and the multiplication of $\deux$ to real vector spaces.
 Then $\ind D$ and $H^+$ determine an element in the $\Pin(2)$-equivariant $KO$-group: 
\[
\alpha = \alpha(X,\fs_{X}) := [\ind{D}] - [H^{+}] \in KO_{\Pin(2)}(B).
\]
Let  $b_{+}(M) := \dim H^{+}(M)$.

If $b_{+}(M) \geq \dim{B} +2$, we can define the {\it  (mod $2$) families Seiberg--Witten invariant}
\[
FSW^{\Z_2}(X,\fs_{X}) \in \Z/2
\]
of $(X,\fs_{X})$
(See Subsection~\ref{subsection: Families Seiberg--Witten invariants}).
The first main result in this paper claims that $FSW^{\Z_2}(X,\fs_{X})$ depends only on  $\alpha(X,\fs_{X})$
which is determined  by the linearization of a family of Seiberg--Witten equations:

\begin{Theorem}[(Theorem~\ref{thm:main})]
\label{theo: main intro}
Let $M_1$ and $M_2$ be oriented closed smooth $4$-manifolds with spin structures 
$\fs_1$ and $\fs_2$ respectively. Assume   the following conditions:
\begin{itemize}
	\item[$\bullet$] $b_1(M_1)=b_1(M_2)=0$, $\ b_+(M_1)=b_+(M_2)\geq \dim B + 2$.
	\item[$\bullet$] $-\dfrac{\sign(M_i)}{4} -1-b_+(M_i)+\dim B =0$ $\ (i=1,2)$.
\end{itemize}	
For $i=1,2$, let $X_i\to B $ be a smooth fiber bundle whose  fiber is  $M_i$ equipped with  a global spin structure $\fs_{X_i}$
 modelled on $\fs_i$.
 
If $\alpha(X_{1},\fs_{X_{1}})=\alpha(X_{2},\fs_{X_{2}})$ holds  in $KO_{\Pin(2)}(B)$, then the equality
\[
FSW^{\Z_2}(X_1,\fs_{X_1}) = FSW^{\Z_2}(X_2,\fs_{X_2})
\]
holds.
\end{Theorem}

In general,
 it is not easy  to calculate $FSW^{\Z_2}(X,\fs)$, since it  is defined  by counting the solutions to a system of the non-linear partial differential equations.
Compared with $FSW^{\Z_2}(X,\fs)$, the linearized data $\alpha(X,\fs_{X})$ is easier to handle.
This allows us to obtain some interesting applications described below.

Combining the rigidity result Theorem~\ref{theo: main intro} 
with a non-vanishing theorem for a specific family of $4$-manifolds  in \cite{BK}, we can obtain non-vanishing of 
the families Seiberg--Witten invariants for some class of families.
This non-vanishing result and a family version of the  argument in \cite{FKM} give us a family version of 10/8-type inequality as follows.
Let $\ell$ be the unique non-trivial real line bundle over $S^1$,
 and $\pi_i\colon T^n=S^1\times\cdots\times S^1\to S^1$ be the projection to the $i$-th  component.
Let us define the real vector bundle $\xi_n$ over $T^n$ by
\[
\xi_n=\pi_1^*\ell\oplus\cdots\oplus \pi_n^*\ell.
\]

\begin{Theorem}[(Corollary~\ref{cor: cor2 revised})]
\label{theo: cor2 revised intro}
Let $M$ be a $4$-manifold with $\sign(M)=-16$ and $b_1(M)=0$.
Let $\fs$ be a spin structure on $M$ and $f_1,\ldots,f_n$ be self-diffeomorphisms on $M$ whose supports $\supp{f}_{1}, \ldots, \supp{f}_{n}$ are mutually disjoint.
Let $H^+\to T^n$ be the bundle of $H^+(M)$ associated to the multiple mapping torus of $f_{1}, \ldots, f_{n}$.
Suppose that each of $f_1,\ldots,f_n$ preserves $\fs$ and that there exists a non-negative integer $a$ such that
\[
H^+ \cong \xi_n\oplus \underline{\R}^a,
\]
where $\underline{\R}^a$ denotes the trivial bundle over $T^{n}$ with fiber $\R^{a}$.
Then the inequality
\begin{align}
b_+(M)\geq n+3
\label{eq: intro 10/8 diffeoms}
\end{align}
holds.
\end{Theorem}

Let us give two remarks on Theorem~\ref{theo: cor2 revised intro}.

\begin{Remark}
Denote by $K3$ the underlying smooth 4-manifold of a $K3$ surface.
Recall that $K3$ admits no diffeomorphisms reversing orientation of $H^{+}(K3)$, which was shown first by Donaldson~\cite{MR1066174}, and later proven also using Seiberg--Witten invariant (for example, see the proof of \cite[Theorem~3.3.28]{MR1787219}).
This fact follows also from the case that $n=1$ and $M=K3$ in Theorem~\ref{theo: cor2 revised intro}, and therefore Theorem~\ref{theo: cor2 revised intro} can be regarded as a generalization of this fact.
\end{Remark}

\begin{Remark}
One may check that the inequality \eqref{eq: intro 10/8 diffeoms} is sharp as follows.
Let us consider the $4$-manifold $M = K3\#n S^{2} \times S^{2}$.
Let $f_{1}, \ldots, f_{n}$ be copies on $n S^{2} \times S^{2}$ of an orientation-preserving diffeomorphism on $S^{2} \times S^{2}$ which reverses orientation of $H^{+}(S^{2} \times S^{2})$ and has a fixed disk.
Then $f_{1}, \ldots, f_{n}$ have mutually disjoint supports, and each of them reverses orientation of $H^{+}(M)$.
Moreover,  the bundle $H^{+} \to T^{n}$ associated to the multiple mapping torus of $f_{1}, \ldots, f_{n}$ is isomorphic to $\xi_{n} \oplus \underline{\R}^3$.
Therefore $M$ and $f_{1}, \ldots, f_{n}$ satisfy the all assumptions in Theorem~\ref{theo: cor2 revised intro}.
Since $b_{+}(M)=n+3$, this example ensures that the inequality \eqref{eq: intro 10/8 diffeoms} is sharp.

Theorem~\ref{theo: cor2 revised intro} claims that, even when $H^+$ for given $f_{1}, \ldots, f_{n}$ is just {\it stably} equivalent to the above example, still one cannot eliminate the part corresponding to ``$nS^{2} \times S^{2}$''.
\end{Remark}

%Let us restate a question whether
%a given topological action of a group $G$ on a topological (but smoothable) $4$-manifold $M$ is non-smoothable
%in terms of algebraic topology. It asks
%whether
% there exists a smooth structure on $M$ such that the homomorphism
%  $\rho : G \to \Homeo(M)$ into the homoemorphism group 
 %  corresponding to the group action factors through the diffeomorphism group $\Diff(M)$ with respect to the smooth structure via the inclusion $\Diff(M) \hookrightarrow \Homeo(M)$:
%\begin{align*}
%\xymatrix{
%     & \Diff(M)\ar@{^{(}->}[d]\\
%    G \ar@{-->}[ru] \ar[r]_-{\rho}  &  \Homeo(M).
%    }
%\end{align*}

As applications of Theorem~\ref{theo: cor2 revised intro} and its generalization Theorem~\ref{thm:cor2}, we shall present two non-smoothability results: non-smoothable families and non-smoothable actions.
First we describe a background of our study of non-smoothable families.
One of motivations of this study is comparison between the diffemorphism and homeomorphism groups of a given manifold.
For lower-dimensional manifolds, it is known that there is no essential difference between these two groups from a homotopical point of view:
for an arbitrary orientable closed manifold $M$ of dimension $\leq 3$, the inclusion 
\begin{align}
\Diff(M) \hookrightarrow \Homeo(M)
\label{eq: Diff to Homeo}
\end{align}
is known to be a weak homotopy equivalence, where
$\operatorname{Homeo}(M)$ and $\operatorname{Diff}(M)$ are equipped with $C^{0}$-Whitney topology and the $C^{\infty}$-Whitney topology, respectively.
(One can check this fact directly in both cases of dimension $1$ and of dimension $2$ with  genus less than $2$.
The case of dimension $2$ and genus greater than $1$ can be reduced to a consideration about the mapping class groups (see, for example, \cite{MR2850125}).
For the $3$-dimensional case, see \cite{MR562642}.)
The dimension $4$ is the smallest dimension in which the inclusion \eqref{eq: Diff to Homeo} may not be a weak homotopy equivalence.

Here we summarize known results in dimension $4$.
First, Donaldson's result on his polynomial invariant (\S~V\hspace{-.1em}I~(i) of \cite{MR1066174}) and Quinn's result (1.1~Theorem in \cite{MR868975}) imply that the natural map $\pi_{0}(\Diff(K3)) \to \pi_{0}(\Homeo(K3))$ is not a surjection.
This follows also from a property of the Seiberg--Witten invariant (see \cite[Theorem~3.3.28]{MR1787219}, for example).
Using Morgan and Szab\'{o}'s rigidity result on the Seiberg--Witten invariant \cite{MS}, one may also show that $\pi_{0}(\Diff(M)) \to \pi_{0}(\Homeo(M))$ is not a surjection also for a homotopy $K3$ surface $M$.
Ruberman \cite{MR1734421} gave the first example of $4$-manifolds $M$ for which $\pi_{0}(\Diff(M)) \to \pi_{0}(\Homeo(M))$ are not injections.
Ruberman's work is based on $1$-parameter families of Yang--Mills anti-self-dual equations, and this is the first striking application of gauge theory for families.
Later, Baraglia and the second author \cite{BK} generalized Ruberman's result using $1$-parameter families of Seiberg--Witten equations, and it was confirmed that $\pi_{0}(\Diff(M)) \to \pi_{0}(\Homeo(M))$ is not an injection for $M=n(K3 \# S^{2} \times S^{2})$ with $n \geq 2$ or $M=2n\CP^{2} \# m(-\CP^{2})$ with $n \geq 2, m \geq 10n+1$.
As a totally different approach, Watanabe~\cite{Watanabe} showed that $\pi_{1}(\Diff(S^{4})) \to \pi_{1}(\Homeo(S^{4}))$ is not an injection using Kontsevich's characteristic classes for sphere bundles.

In this paper, we propose a new approach to the comparison problem between $\Diff(M)$ and $\Homeo(M)$ in dimension $4$.
Our strategy is that, developing gauge theory for families, we shall obtain a constraint on a smooth fiber bundle of a $4$-manifold, and detect a {\it non-smoothable topological families} of smooth $4$-manifolds.
The existence of such a family implies that $\Diff(M) \hookrightarrow \Homeo(M)$ is not a weak homotopy equivalence for the fiber $M$.
Here let us clarify the meaning of ``non-smoothable'' topological families.
Let $M$ be an oriented topological manifold admitting a smooth structure, 
$B$ be a smooth manifold and $M \to X \to B$ be a fiber bundle whose structure group is in $\operatorname{Homeo}(M)$.
We say that the bundle $X$ is {\it non-smoothable as a family} or $X$ has {\it no smooth reduction} if for any smooth structure on $M$ there is no reduction of the structure group of $X$ to $\operatorname{Diff}(M)$ via the inclusion $\operatorname{Diff}(M) \hookrightarrow \operatorname{Homeo}(M)$.
Namely, we say that $X$ is non-smoothable as a family
if there is no lift of the classifying map $\varphi : B \to B\Homeo(M)$ of $X$ to $B \Diff(M)$ along the natural map $B \Diff(M) \to B \Homeo(M)$
with respect to any smooth structure on $M$:
\begin{align*}
\xymatrix{
     & B\Diff(M)\ar[d]\\
    B \ar@{-->}[ru] \ar[r]_-{\varphi}  &  B\Homeo(M).
    }
\end{align*}

%Note that non-smoothable families are harder to detect than non-smoothable actions.
%For a topological group action $\rho : G \to \Homeo(M)$, if one can show that $\varphi := B\rho : BG \to B\Homeo(M)$ is non-smoothable as a family, then we can deduce that $\rho$ is non-smoothable as a group action, but there is no backward way in general.
%This is clear if we consider just a single homeomorphism $f$ and its mapping torus over $S^{1}$.
%The $\Z$-action by $f$ is a non-smoothable group action if and only if $f$ is not a diffeomorphism for any smooth structure, and the mapping torus by $f$ is a non-smoothable family if and only if
%$f$ is not isotopic to a diffeomorphism for any smooth structure.

Now we can describe our non-smoothability results.
Let $-E_{8}$ denote the (unique) closed simply connected oriented topological 4-manifold whose intersection form is the negative definite $E_{8}$-lattice.
For a subset $I=\{i_{1},\ldots,i_{k}\} \subset \{1,2,\ldots,m\}$ with cardinality $k$,
denote by 
$
T^{k}_{I}
$
the $k$-torus embedded in the $m$-torus $T^{m}$ defined as the product of the $i_{1}, \ldots, i_{k}$-th $S^{1}$-components.
The following theorem claims that there exist non-smoothable families over the torus $T^{n}$ for $n \in \{1, \ldots,4\}$ whose fibers are the topological (but smoothable) $4$-manifolds $2(-E_{8}) \# m S^{2}\times S^{2}$ with $m = n+2$.
Moreover, we shall ensure that the total spaces of the families are smoothable as manifolds.

\begin{Theorem}[(Theorem~\ref{theo: application to nonsmoothable family})]
\label{theo: application to nonsmoothable family intro}
	Let $3\leq m \leq 6$.
	Let $M$ be the topological (but smoothable) $4$-manifold defined by
	\[
	M = 2(-E_{8}) \# m S^{2}\times S^{2}.
	\]
	Then there exists a $\operatorname{Homeo}(M)$-bundle \[
	M \to X \to T^{m}
\]
over the $m$-torus satisfying the following properties: 
	let $I=\{i_{1},\ldots,i_{k}\}$ be an arbitrary subset of $\{1,2,\ldots,m\}$ with cardinality $k$.
	\begin{itemize}
		\item[$\bullet$] The total space $X$ admits a smooth manifold structure.
		\item[$\bullet$] If $k\leq m-3$,  
		the restricted family
		\[
		X|_{T^{k}_{I}} \to T^{k}_{I}
		\]
		admits a reduction to $\Diff(M)$ for some smooth structure on $M$.
		\item[$\bullet$] If $m-2\leq k\leq m$, 
		the restricted family
		\[
		X|_{T^{k}_{I}} \to T^{k}_{I}
		\]
		has no reduction to $\Diff(M)$ for any smooth structure on $M$.
	\end{itemize}
\end{Theorem}

Non-smoothablity as families in Theorem~\ref{theo: application to nonsmoothable family intro} is detected by Theorem~\ref{thm:cor2}, which generalizes Theorem~\ref{theo: cor2 revised intro}.
To apply this theorem, we need to calculate the Dirac index bundle.
To do this, we shall use (a variant of) the celebrated Novikov's theorem on topological invariance of rational Pontrjagin classes.
Smoothability as manifolds is verified using Kirby--Siebenmann theory.

\begin{Remark}
As noted above, for a homotopy $K3$ surface $M$, it was shown that $\pi_{0}(\Diff(M)) \to \pi_{0}(\Homeo(M))$ is not a surjection.
Let $M \to X \to S^{1}$ be the mapping torus for a representative of a non-zero topological isotopy class in the cokernel of $\pi_{0}(\Diff(M)) \to \pi_{0}(\Homeo(M))$.
Then $X$ is an example of a non-smoothable family.
Theorem~\ref{theo: application to nonsmoothable family intro} contains this simplest example $M \to X \to S^{1}$.
To our knowledge, non-smoothable familis over higher-dimensional base spaces and the problem of smoothing of the total spaces have not been discussed so far.
\end{Remark}

From the last property of $X$ in Theorem~\ref{theo: application to nonsmoothable family intro}, non-smoothability as a family,
we immediately obtain:

\begin{Corollary}
\label{cor: difference between diff and homeo intro}
For $0\leq n\leq 3 $, let $M$ be a smooth $4$-manifold which is homeomorphic to $K3\# n S^{2} \times S^{2}$.
Then the inclusion
\[
\Diff(M) \hookrightarrow \Homeo(M)
\]
is not a weak homotopy equivalence.
\end{Corollary}

Here are three remarks on Corollary~\ref{cor: difference between diff and homeo intro}.

\begin{Remark}
The result in the case that $n=0$ of Corollary~\ref{cor: difference between diff and homeo intro} follows also from the combination of Morgan--Szab\'{o}~\cite{MS} and Quinn~\cite{MR868975},
however, to our best knowledge, the result in the case that $n >0$ are new even when $M$ is diffeomorphic to $K3\# n S^{2} \times S^{2}$.
To see that the case that $n=0$ follows from \cite{MS} and \cite{MR868975}, consider the unique spin structure on a smooth $4$-manifold homeomorphic to $K3$.
This $4$-manifold has non-zero Seiberg--Witten invariant for the spin structure by \cite{MS}, and from this we can deduce that there does not exist an orientation-preserving diffeomorphism which reverses the orientation of $H^{+}$.
\end{Remark}

\begin{Remark}
We note that there exist exotic $K3\#nS^{2} \times S^{2}$ not only for $n=0$.
To see the existence of exotic $K3\#nS^{2} \times S^{2}$ for some positive $n$, we may use a result of Park--Szab\'{o}~\cite{MR1653371}.
By Theorem~1.1 or Proposition~3.2 of \cite{MR1653371},
we may ensure that there exist exotic $K3\#2k(S^{2} \times S^{2})$ for all $k>0$.
\end{Remark}

\begin{Remark}
By results of Wall~\cite{Wall} and Quinn~\cite{MR868975}, any algebraic automorphism of the intersection form of $K3\# nS^{2} \times S^{2}$ is realized both by a homeomorphism and a diffeomorphism for $n \geq 1$.
Therefore we cannot find any difference between $\Diff(K3\# n S^{2} \times S^{2})$ and $\Homeo(K3\# n S^{2} \times S^{2})$ only using realizability of an automorphism on the intersection form.
\end{Remark}

Furthermore, combining Theorem~\ref{theo: application to nonsmoothable family intro} with an observation relating to results of Wall~\cite{Wall} and Quinn~\cite{MR868975} (Proposition~\ref{prop: application to nonsmoothable family 2}),
we can also obtain information about a sort of quotient of $\Homeo(M)$ divided by $\Diff(M)$ for $M=K3 \# S^{2} \times S^{2}$.
To be precise, since $\Diff(M)$ is not closed in $\Homeo(M)$ with respect to a natural topology such as the $C^{0}$-topology, we consider the {\it homotopy quotient} 
\[
\Homeo(M) \sslash \Diff(M) := \left(E\Diff(M) \times \Homeo(M)\right)/\Diff(M).
\]

\begin{Theorem}[(Corollary~\ref{cor: the fun group of quotient})]
Let $M=K3\# S^{2} \times S^{2}$.
Then we have 
\[
\pi_{1}(\Homeo(M) \sslash \Diff(M)) \neq 0.
\]
\end{Theorem}

As another application,
on the topological $4$-manifold $2(-E_{8}) \# m S^{2}\times S^{2}$ with $m \geq 3$,
we shall construct non-smoothable $\Z^{m-2}$-actions.
Note that the $4$-manifold $2(-E_{8}) \# m S^{2}\times S^{2}$ is homeomorphic to $K3\#(m-3)S^{2} \times S^{2}$
and hence  admits a smooth structure.

\begin{Theorem}[See Theorem~\ref{theo: application to nonsmoothable action}]
\label{theo: application to nonsmoothable action intro}
Let $m \geq 3$.
The topological  (but smoothable) $4$-manifold $M$ defined by
\[
M = 2(-E_{8}) \# m S^{2}\times S^{2}
\]
admits commuting self-homeomorphisms $f_{1}, \ldots, f_{m}$ satisfying the following properties:
let $I=\{i_{1},\ldots,i_{k}\}$ be an arbitrary subset of $\{1,2,\ldots,m\}$ with cardinality $k$.
\begin{itemize}
\item[$\bullet$] If $k\leq m-3$, then there exists a smooth structure on $M$ such that $f_{i_{1}}, \ldots, f_{i_{k}}$ are diffeomorphisms with respect to the smooth structure.
\item[$\bullet$] If $k\geq m-2$, then 
there exists no smooth structure on $M$ such that all 
$f_{i_{1}}, \ldots, f_{i_{k}}$ are diffeomorphisms with respect to the smooth structure.
\end{itemize}
\end{Theorem}

Non-smoothable group actions on $4$-manifolds has been studied by many authors.
The main tool to detect them is equivariant gauge theory, but the third author of this paper found that  gauge theory for families can be also used to study non-smoothable actions in \cite{MR2644908}, and  this direction was developed by 
 Baraglia~\cite{Baraglia}.
The proof of Theorem~\ref{theo: application to nonsmoothable action intro} based on a technique different from \cite{MR2644908} and \cite{Baraglia}, and the result itself is new.
We will compare the result of Theorem~\ref{theo: application to nonsmoothable action intro} with previous research in Remark~\ref{rem: precious works on nonsmoothable actions} in detail.

As a further research direction, once one can establish a Bauer--Furuta version of the gluing result in \cite{BK}, then we may get results on non-smoothable actions and families on any  spin $4$-manifold with signature $-32$ following  the same strategy of this paper.
However, in fact, we are also considering to develop a way to deal with more general signature and $b_{+}$.

A brief outline of the contents of this paper is as follows.
In Section~\ref{section: Monopole map}, we shall recall some materials of the families Seiberg--Witten invariant.
In Section~\ref{section: Spin families}, we shall discuss when the tangent bundle along fibers admits a fiberwise spin structure.
In Section~\ref{section: Main Theorem}, we shall prove the main results in this paper, the rigidity theorem, and its consequences, such as a $10/8$-type inequality.
In Section~\ref{section: Applications}, we shall give two applications, non-smoothable actions and families, of the results given in Section~\ref{section: Main Theorem}.
Sections~\ref{section: Calculation of the index bundle} and \ref{section: Smoothing of the total spaces} are devoted to prove some results needed to establish the applications in Section~\ref{section: Applications}.
The main tools in Sections~\ref{section: Calculation of the index bundle} and \ref{section: Smoothing of the total spaces} are Kirby--Siebenmann theory, and arguments there may be of independent interest even forgetting gauge-theoretic context.
In Sections~\ref{section: Calculation of the index bundle}, we shall calculate the Dirac index bundle.
More precisely, we shall give a few sufficient conditions for families of spin $4$-manifolds to have trivial index bundles.
In Section~\ref{section: Smoothing of the total spaces}, we shall show the smoothability as manifolds of the total spaces of the non-smoothable families given in Section~\ref{section: Applications}.

\begin{Addendum}
After the first version of this paper appeared on arXiv, David Baraglia informed the second author about a draft of his paper \cite{Baraglia2}.
Adapting the construction of examples of families in this paper for his constraints on families of $4$-manifolds, he generalizes Corollary~\ref{cor: difference between diff and homeo intro} of this paper as Theorem~1.8 and Corollary~1.9 of \cite{Baraglia2}.
We note that his way to prove these results is different from ours:
to prove Theorem~1.8 and Corollary~1.9, Baraglia used a family version of Donaldson's diagonalization theorem (corresponding to Theorems~1.1, 1.2 in \cite{Baraglia2}),
on the other hand, we use a family version of the 10/8-inequality to prove Corollary~\ref{cor: difference between diff and homeo intro}.
\end{Addendum}

\begin{acknowledgement}
The authors would like to thank Yosuke Morita for helpful conversations on Proposition~\ref{prop: application to nonsmoothable family 2}.
The authors would also appreciate giving useful comments on algebraic topology to Yuli Rudyak and Daisuke Kishimoto. 
The authors also wish to thank Mikio Furuta for pointing out a mistake in an earlier version of this preprint, and also wish Ko Ohashi and Yukio Kametani for helpful comments and discussion.
Tsuyoshi Kato was supported by JSPS Grant-in-Aid for Scientific Research (B) No.17H02841 and 
JSPS Grant-in-Aid for Scientific Research on Innovative Areas (Research in a proposed research area) No.17H06461.
Hokuto Konno was partially supported by JSPS KAKENHI Grant Numbers 16J05569, 17H06461, 19K23412 and 
the Program for Leading Graduate Schools, MEXT, Japan, and further, a part of this work was done under RIKEN Special Postdoctoral Researcher program.
Nobuhiro Nakamura was supported by JSPS Grant-in-Aid for Scientific Research (C) No.19K03506.
\end{acknowledgement}

\section{Families Seiberg--Witten invariant}
\label{section: Monopole map}

In this section, we shall recall some materials of the families Seiberg--Witten invariant.
In particular, we shall recall an interpretation of the families Seiberg--Witten invariant as a kind of mapping degree in Subsection~\ref{subsection: Families Seiberg--Witten invariants}.
%%%%%%%%%%%%
%
%
\subsection{Seiberg--Witten equations with $j$-action}\label{subsection:SWj}

First we recall some basics of the Seiberg--Witten equations on a spin $4$-manifold in the unparameterized setting.
A special feature of the Seiberg--Witten equations on a spin $4$-manifold is that the equations have an extra symmetry, written as the ``$j$-action'', compared with the Seiberg--Witten equations on a general $\Spinc$ $4$-manifold.
We refer the readers to \cite{MR1367507} for the generality of the Seiberg--Witten equations,
and \cite{MR1839478,Furuta, BF} for the monopole maps on spin $4$-manifolds.

Let $M$ be a closed Riemannian $4$-manifold with a spin structure $\fs$.
Let $S=S^+\oplus S^-$ be the spinor bundle.
Note that $S$ has a quaternionic structure, in particular, the multiplication of $j \in \mathbb{H}$ is defined.
The $j$-action is anti-linear with respect to the complex structure of $S$.
Let us abbreviate the tangent bundle $TM$ to $T$ and identify $T$ with the cotangent bundle $T^*M$ by the metric.
Let $C(T)$ be the Clifford bundle of $T$. 
As a vector bundle, the bundle $C(T)$ is identified with the bundle of exterior product, $\Lambda^*T$.
The Clifford multiplication is given by a bundle morphism
\[
\rho\colon \Lambda^*T\to\End_{\R}( S).
\]
Namely, for $v\in \Lambda^*T_x$,  $\rho(v)$ is an endomorphism of $S_x$.
Here $T_x$, $S_x$ are the fibers over $x$.
The spinor bundle has a $\Z_2$-grading $S=S^+\oplus S^-$, and we have also $\Lambda^*T=\Lambda^{\text{even}}T\oplus \Lambda^{\text{odd}}T$.
If $v$ is in $\Lambda^{\text{even}}T$, $\rho(v)$ preserves the $\Z_2$-grading of $S^+\oplus S^-$.
If $v\in \Lambda^{\text{odd}}T$, then $\rho(v)$ switches  $S^+$ and $S^-$.
Note that the Clifford multiplication $\rho(v)$ commutes with the $j$-action:
\[
\rho(v)j = j\rho(v).
\]

The complexified Clifford multiplication is also defined
\[
\rho\colon\Lambda^*T\otimes_\R\C\to\End_{\C}(S),
\] 
which {\it anti-commutes} with the $j$-action, that is, for $v\otimes c \in \Lambda^*T\otimes_\R\C$, we have
\[
\rho(v\otimes c) j = j\rho(v\otimes \bar{c}).
\]

The Levi-Civita connection on $T$ induces a spin connection $\nabla_0$ on $S$, and the spin Dirac operator 
\[
D_0\colon\Gamma(S^+)\to\Gamma(S^-)
\]
is defined by
\[
D_0=\sum_i\rho(e_i)(\nabla_0)_{e_i},
\]
where $\{e_i\}$ is a local orthonormal frame of $T$.
Then $D_0$ commutes with $j$.
Note that the spin connection $\nabla_0$ induces a trivial flat connection $A_0$ on $L_0=\det S^+ \cong M\times \C$.
The $j$-action on $S^+$ induces the $j$-action on $L_0$  given by complex conjugation.
Let $A$ be a $\U(1)$-connection  on $L_0$.
If we write $A$ as $A=A_0+a$ for an imaginary-valued $1$-form $a\in i\Omega^1(M)$, then the $j$-action on $L_0$ induces the $j$-action on $\U(1)$-connections given by
\[
j\cdot A = A_0-a.
\]
For a $\U(1)$-connection $A=A_0+a$,  we have a unique $\Spinc(4)$-connection $\nabla_0 + \frac{a}2$ on $S$ which induces the Levi-Civita connection on $T$ and $A$ on $L_0$.
We have the Dirac operator associated with $\nabla_0 + \frac{a}2$  as follows:
\[
D_A\phi=D_0\phi+\frac12\rho(a)\phi.
\]
In fact, $D_A$ is a Dirac operator on the $\Spinc$-structure associated to the spin structure $\fs$. 
Note that $D_A$ is $j$-equivariant:
\[
D_{(j\cdot A)}(j\phi) = jD_A\phi.
\]

As mentioned above, $\rho(v)$ for even degree $v$ preserves the components $S^\pm$.
In particular, it can be seen that $\rho(v)$ for a self-dual $2$-form $v$ is an endomorphism of $S^+$, that is, $\rho(\Lambda_+^2T\otimes\C)\subset \End_{\C}(S^+)$.
\begin{Proposition}[({\cite[Lemma 4.1.1]{MR1367507}})]
	For $\phi\in\Gamma(S^+)$, define the traceless endomorphism $q(\phi)$ by
	\[
	q(\phi)=\phi\otimes\phi^*-\frac{|\phi|^2}{2}\id.
	\]
Then $q(\phi)$ can be identified with a section of $\Lambda_+^2(T)\otimes i\R$.
\end{Proposition}
Now we can write down the Seiberg--Witten equations:
\begin{equation}\label{eq:SW}
\left\{
\begin{aligned}
D_{A}\phi=&0,\\
F_{A}^+ =&q(\phi),
\end{aligned}\right.
\end{equation}
where $F_A^+$ is the self-dual part of the curvature of $A$.
If we write $A=A_0+a$, the equations \eqref{eq:SW} are rewritten as  
\[\left\{
\begin{aligned}
D_{A_0}\phi+\frac12\rho(a)\phi=&0,\\
d^+a -q(\phi) =&0.
\end{aligned}\right.
\]
As we have already seen, the first equation is $j$-equivariant.
Since we have
\[
F_{jA}=F_{A_0-a}=-da,\quad q(j\phi)=-q(\phi),
\]
the second equation is also $j$-equivariant.

The gauge transformation group $\G=\Map(M,\U(1))$ acts on 
the space of $\U(1)$-connections of $L_{0}$ and positive spinors by
\[
u(A,\phi) = (A-2u^{-1}du,u\phi)
\]
for $u \in \G$.
The Seiberg--Witten equations \eqref{eq:SW} are $\G$-equivariant.
The gauge action anti-commutes with the $j$-action:
\[
u(x)j=j\overline{u(x)}\quad\text{ for }x\in M.
\]

The moduli space of solutions to the Seiberg--Witten equations is the set of gauge equivalence classes of solutions:
\[
\M=\{\text{solutions to \eqref{eq:SW}}\}/\G.
\] 
Roughly speaking, the Seiberg--Witten invariant is defined by ``counting of $\#\M$".
Furthermore, the number ``$\#\M$" can be interpreted as the ``mapping degree" of a map, called the {\it monopole map}, whose zero set is essentially $\M$.
(Precise meaning of these will be explained in Subsection~\ref{subsection: Families Seiberg--Witten invariants} in a parametrized setting.)

The monopole map is defined by
\begin{equation}\label{eq:monopolemap}
\begin{gathered}
m\colon i\Omega^1(M)\oplus \Gamma(S^+)\to i(\Omega^0_*\oplus\Omega^+)(M)\oplus \Gamma(S^-),\\
m(a,\phi)= (d^*a, d^+a - q(\phi), D_{A_0}\phi+\frac12\rho(a)\phi),
\end{gathered}
\end{equation}
where $\Omega^0_*(M)=\Ima(d^*\colon\Omega^1(M)\to\Omega^0(M))$.
The  map $m$ is decomposed into the sum $m=l+c$ of the linear map $l=(d^*,d^+, D_{A_0})$ and the quadratic part $c$ given by $c(a,\phi)=(0,-q(\phi),\frac12\rho(a)\phi)$.
For the purpose of carrying out a suitable analysis, we take the $L^2_k$-completion ($k\geq 4$) of the domain, and the $L^2_{k-1}$-completion of the target, and extend $m$ to the completed spaces.
Denote by $\mathcal{U}^\prime$ and $\mathcal{U}$
the completed domain and target respectively.
Then $m\colon \mathcal{U}^\prime\to\mathcal{U}$ is a smooth map between Hilbert spaces.
The linear part $l$ is a Fredholm map of index
\[
-\frac{\sign(M)}4 + b_1(M) -b_+(M),
\]
and $c$ is a non-linear compact map.

We take the $L^2_{k+1}$-completion of the gauge group $\G$.
Then the $\G$-action is smooth.
The space $\ker (d^*\colon i\Omega^1(M)\to i\Omega^0(M))$ is a global slice of the $\G$-action at $(0,0)$, and we have 
\[
m^{-1}(0) = \{\text{solutions to \eqref{eq:SW}}\}\cap \ker d^*.
\]
The slice $\ker d^*$ still has a remaining gauge symmetry. 
Let $\mathrm{Harm}(M,S^1)$ be the kernel of the composition of the maps
\[
L^2_{k+1}(\Map(M,S^1))\overset{d}{\to}L^2_k(i\Omega(M))\overset{d^*+d^+}{\to }L^2_{k-1}(i(\Omega^0\oplus\Omega^+)(M)).
\]
Denote by $\mathrm{Harm}(M,S^1)$ the space of harmonic maps from $M$ to $S^{1}$.
Then $m$ is $\mathrm{Harm}(M,S^1)$-equivariant, and we have
\[
\M = m^{-1}(0)/\mathrm{Harm}(M,S^1).
\]
We also have an identification
\[
\mathrm{Harm}(M,S^1) \cong S^1\times H^1(M;\Z),
\]
which is obtained by fixing a splitting of  the exact sequence
\[
1\to S^1\to \mathrm{Harm}(M,S^1)\to H^1(M;\Z)\to 0.
\]
The monopole map $m$ is also $j$-equivariant, when $j$ acts on the spaces $i\Omega^*(M)$ of imaginary-valued forms by multiplication of $-1$. 
The $j$ action anti-commutes with the $\mathrm{Harm}(M,S^1)$-action in the sense that 
\[
j(z,a)= (\bar{z}, -a)j
\] 
for $(z,a)\in S^1\times H^1(M;\Z)$. 

Set $\Pin(2) = \langle S^1, j\rangle$, the group generated by $S^{1}$ and $j$ in $\mathbb{H}$.
Assuming $b_1(M)=0$, we have $\mathrm{Harm}(M,S^1) = S^1$, and $m$ is $\Pin(2)$-equivariant. 
Since $\M=m^{-1}(0)/S^1$, the $j$-action  descends to a $\Pin(2)/S^1$-action on $\M$, where $\Pin(2)/S^1=\deux$.

\subsection{Families Seiberg--Witten invariants}
\label{subsection: Families Seiberg--Witten invariants}
%%%%%%%%%%%%

For a given family of $4$-manifolds,
one can define a family version of the Seiberg--Witten invariant by counting the numbers of the parameterized moduli space of the Seiberg--Witten equations.
This invariant can be also interpreted as the mapping degree of a finite dimensional approximation of a family of monopole maps.
In this subsection we shall recall these arguments.
See such as \cite{B,MR1868921,BF,BK2,MR2652709} for references for this subsection.

Let $M$ be a closed oriented smooth $4$-manifold with a spin structure $\fs$, $B$ be a closed smooth connected manifold, and
$\pi\colon X\to B$ be a locally trivial fiber bundle with fibers diffeomorphic to $M$.
We assume that the structure group of $\pi\colon X\to B$ is in the group of orientation-preserving diffeomorphisms of $M$.
In such a case we call $\pi\colon X\to B$ a {\it smooth family} of $M$.
Let $T(X/B)$ be the tangent bundle along the fibers and choose a metric on $T(X/B)$. 
We shall consider the situation that $T(X/B)$ admits a spin structure $\fs_X$ whose restriction on  each fiber is isomorphic to $\fs$.
We call such a spin structure $\fs_X$ a \textit{global spin structure modelled on $\fs$}.
If we start with a  $4$-manifold $M$ with $\Spinc$-structure $\fs^c$,  a \textit{global $\Spinc$ structure modelled on $\fs^c$} is similarly defined.

\begin{Remark}
As above, we say that a topological fiber bundle $X \to B$ is smooth if its structure group has been reduced to $\Diff(M)$ from $\Homeo(M)$.
On the other hand, since we assumed that $B$ is a smooth manifold, another option of the definition of a smooth fiber bundle is a stronger one.
Namely, one might assume some smoothness on the transition functions in the following sense: let $\{g_{\alpha\beta} : U_{\alpha} \cap U_{\beta} \to \Diff(M)\}$ be the transition functions of $X \to B$ for some open covering $M=\bigcup_{\alpha} U_{\alpha}$.
We say that the transition functions are smooth if the map
\[
(U_{\alpha} \cap U_{\beta}) \times M \to M
\]
given by
\[
(b,x) \mapsto g_{\alpha\beta}(b)x
\]
is smooth.
If the transition functions satisfy this smoothness, then the total space $X$ is a smooth manifold and the projection $X \to B$ is a smooth map.
One might call $X \to B$ a smooth fiber bundle only in this case.

However, in fact, it is shown in \cite{MR2574141} that these two
definitions of a smooth fiber bundle are equivalent to each other:
if a topological fiber bundle $X\to B$ over a smooth manifold $B$ has a reduction to $\Diff(M)$, then, after replacing $X$ with an isomorphic bundle, we may assume that the transition functions of $X$ satisfy the smoothness in the above sense.
\end{Remark}

Assume that $b_1(M)=0$ and that a global spin structure $\fs_X$ modelled on $\fs$ is given on $\pi\colon X\to B$.
Let $m\colon\mathcal{U}^\prime\to\mathcal{U}$ be the Sobolev completed monopole map given in \subsecref{subsection:SWj}.
Recall $m$ is $\Pin(2)$-equivariant since $b_1(M)=0$. 
Once we fix a fiberwise Riemannian metric on $X$, we can obtain a family of monopole maps
\[
\tilde{\mu}\colon\mathcal{A}\to\mathcal{C}
\]
by parametrizing the argument in \subsecref{subsection:SWj} over $B$.
Here $\mathcal{A}$ and $\mathcal{C}$ are the Hilbert bundles over $B$ with fibers $\mathcal{U}^\prime$ and $\mathcal{U}$,
 and $\tilde{\mu}$ is a fiber-preserving map whose restriction on each fiber is identified with the monopole map $m$.

It is convenient to trivialize $\mathcal{C}\cong B\times\mathcal{U}$ by Kuiper--Segal's theorem \cite{MR0179792}, \cite{MR0248877}.
Define $\mu\colon\mathcal{A}\to\mathcal{U}$ by the composition of $\tilde\mu$ with the projection $\mathcal{C}\cong B\times\mathcal{U}\to\mathcal{U}$.
The map $\mu$ satisfies the following compactness property.

\begin{Proposition}[({\it cf}.~\cite{BF})]
\label{prop: properness of BF}
	The preimage $\mu^{-1}(O)$ of a bounded set $O\subset\mathcal{U}$ is contained in some bounded disk bundle.
\end{Proposition}

This can be verified straightforwardly extending the argument in \cite[Proposition 3.1]{BF} since we have assumed that $B$ is compact.
By Proposition~\ref{prop: properness of BF}, the map $\mu$ can be extended to the map
\[
\mu^+\colon T\mathcal{A}\to S^{\mathcal{U}},
\]
where $T\mathcal{A}$ is the Thom space and $S^{\mathcal{U}}$ is the one-point completion of $\mathcal{U}$ obtained by collapsing.

The family $\pi\colon X\to B$ induces a vector bundle
\[
\R^{b_{+}(M)} \to H^+ \to B
\]
whose fiber over $b\in B$ is the space $H^+(M_b)$ of harmonic self-dual $2$-forms on $M_b=\pi^{-1}(b)$.
We call $H^+$ the bundle of $H^+(M)$.
The isomorphism class of $H^{+}$ is independent of the choice of fiberwise Riemannian metric on $X$ since the Grassmannian of maximal-dimensional positive definite subspaces of $H^{2}(M;\R)$ is contractible.
We also have the index bundle $\ind D \in KO_G(B)$ of the family of  Dirac operators on the spin family $X\to B$. 
Here we assume $G=\Pin(2)$ or $G=S^1 \subset \Pin(2)$.
Let $L\colon \mathcal{A}\to\mathcal{C}$ be the family of linear parts of $\mu$, which is
 a fiberwise linear map whose  restriction on  each fiber is $l$ in \subsecref{subsection:SWj}.
In Section~\ref{section: Main Theorem}, we will use the element
 $\alpha \in KO_G(B)$ defined by
\[
\alpha:=[\ind L]=[\ind D] -[H^+] \in KO_G(B).
\]
Choose a finite dimensional trivial vector subbundle $F^\prime = \underline{V}=B\times V\subset \mathcal{C}$ so that $F^\prime$ contains the fiberwise cokernel of $L$, and let $F=L^{-1}(F^\prime)$.
Then $\alpha= [F]-[F^\prime]$ holds
and the image of $F$ under $L$ is contained in $\underline{V}$.
On the other hand, the image of $F$ under the nonlinear part $\mu-L$ is not necessarily contained in $\underline{V}$, and
we shall project the image of $\mu$ on $V$.
Let $S(V^\perp)$ be the unit sphere in the orthogonal complement $V^\perp$ of $V$.
The inclusion $S^{V}\to S^{\mathcal{U}}\setminus S(V^\perp)$ is a deformation retract.
Let $\rho_V$ be a retracting map.
A \textit{finite dimensional approximation of the family of monopole maps} is defined by
\[
f_{V}=\rho_V\circ\mu|_{TF}\colon TF\to V.
\]
By \cite{BF}, the above construction defines a well-defined class $[f_V]$ in the stable cohomotopy set
\begin{equation}\label{eq:stab}
\{T(\ind D -H^+),S^0\}^G_{\mathcal{U}} = \underset{W\subset V^\perp}{\colim}[S^W\wedge TF,S^W\wedge S^V]^G.
\end{equation}
We call the class $[f_V]$ the {\it Bauer--Furuta invariant} or the \textit{stable cohomotopy Seiberg--Witten invariant of the family $\pi\colon X\to B$}. 

In the case when $G=S^1$, we shall define the (mod $2$) degree homomorphism
\[
\deg\colon \{TF,S^V\}^{S^1}_{\mathcal{U}}\to\Z_2 %\quad [f]\mapsto \deg f\mod 2
\] 
below, provided that
\[
d:=-\dfrac{\sign(M)}{4} -1-b_+(M_i)+\dim B =0.
\]
The condition $d=0$ is equivalent to that
\[
\rank F +\dim B -\dim V =1,
\]
and therefore the preimage of $f_{V}$ is $1$-dimensional.
For a finite dimensional approximation $f_V$ with sufficient large $V$ such that \eqref{eq:stable} below holds, we let 
\[
FSW^{\Z_2}(X,\fs_X) = \deg [f_V].
\]
This definition coincides with the one in \cite[Section 2]{BK2}.

When $G=S^1$, the universe $\mathcal{U}$ consists of $\C$ on which $S^1$ acts by multiplication and the trivial real $1$-dimensional $S^1$-representation $\R$ as irreducible summands. 
In this case, the stable cohomotopy set \eqref{eq:stab} is a group.
The bundle $F$ is an $S^1$-equivariant bundle with fiber $\C^{x+2a}\oplus \R^y$ over $B$ and $V=\C^x\oplus\R^{y+b}$, where $x$, $y$ are nonnegative integers and 
\[
a = -\frac{\sign(M)}{16},\quad b=b_+(M).
\]

When $G=\Pin(2)$, the universe $\mathcal{U}$ consists of $\HH$ on which $\Pin(2)$ acts by multiplication and the real $1$-dimensional nontrivial $\Pin(2)$-representation $\tilde{\R}$ as irreducible summands.
In order to distinguish this universe from the one in the case when $G=S^1$, this $\Pin(2)$-universe is denoted by $\mathcal{U}^\prime$ henceforth.
In this case, $F$ is a $\Pin(2)$-equivariant bundle with fiber $\HH^{x+a}\oplus \tilde{\R}^y$ over $B$ and $V=\HH^x\oplus\tilde{\R}^{y+b}$ for some nonnegative integers $x$, $y$.

Suppose that
\begin{equation}\label{eq:dim}
b_+(M)\geq \dim B+2. 
\end{equation}
Let $Ci$ be the mapping cone of the inclusion $i \colon TF^{S^1}\hookrightarrow TF$ of the $S^1$-fixed point set.
We have a long exact sequence associated with the cofiber sequence $TF^{S^1}\to TF\to Ci$
\[
\{S^1\wedge TF^{S^1},S^V\}^G_{\mathcal{U}}\to \{Ci,S^V\}^G_{\mathcal{U}}\to \{TF,S^V\}^G_{\mathcal{U}}\to\{ TF^{S^1},S^V\}^G_{\mathcal{U}}.
\]
Since both of the first and the last terms are trivial by the assumption \eqref{eq:dim}, the cohomotopy invariant $[f_V]$ can be regarded as an element of $\{Ci,S^V\}^G_{\mathcal{U}}$.
Following \cite{B}, we let $[TF,S^V]^{G}_q$ be the  set of homotopy classes of maps $g\colon TF\to S^V$ such that $g|_{TF^{S^1}}=f_{V}|_{TF^{S^1}}$.
Then the condition \eqref{eq:dim} implies a natural bijective correspondence
\[
[TF,S^V]^{G}_q\cong [TF/TF^{S^1},S^V]^{G}.
\] 
Since $F$ is a finite dimensional $\Pin(2)$-equivariant vector bundle over a smooth compact connected manifold $B$ and $V$ is a finite dimensional $\Pin(2)$-representation, the Thom space $TF$ and $S^V$ can be equipped with structures of  $\Pin(2)$-equivariant CW complexes.  
By the equivariant Freudenthal suspension theorem(\cite[Chapter II, (2.10)]{tom Dieck}), we can choose sufficiently large $F$, $V$ satisfying that
\begin{equation}\label{eq:stable}
\{T(\ind D-[H^+]), S^0\}^{G}_{\mathcal{U}}\cong [TF,S^V]^{G}_q.
\end{equation}

To analyse $[TF,S^V]^{G}_q$, it is convenient to use the equivariant obstruction theory. (See Appendix.)
Let $U=(TF/S^1)\setminus N(TF^{S^1})$, where $N(TF^{S^1})$ is an equivariant tubular neighborhood of $TF^{S^1}$ in $TF/S^1$.
Then $U$ is a (possibly nonorientable) manifold with boundary.
Set $k=\dim S^V$.
The condition $d=0$ implies that $\dim U=k$.
Note that $(TF/S^1)\setminus TF^{S^1}$ is $\Z_2$-equivariantly homotopic to $U$, where $\Z_{2}=\Pin(2)/S^1$.
Then we have the following identifications:
\begin{gather*}
H^k_{S^1}(TF,TF^{S^1};\pi_kS^V)\cong H^k(TF/S^1,TF^{S^1};\Z)\cong H^k(U,\partial U; \Z), \\%\cong \Z \text{ or }\Z_2,\\
H^k_{\Pin(2)}(TF,TF^{S^1};\pi_kS^V)\cong H^k_{\Z_2}(TF/S^1,TF^{S^1};\pi_kS^V)\cong H^k_{\Z_2}(U,\partial U; \pi_kS^V),
\end{gather*}
where $H^k_G(\cdot,\cdot;\pi_kS^V)$ are the Bredon cohomology groups with coefficient $G$-module $\pi_kS^V$. 
Let us recall the following facts:
\begin{itemize}
	\item[(F1)] If we choose $\beta_0 \in [TF,S^V]^{G}_q$ for each $G=\Pin(2)$ and $S^1$, then  the correspondece $\beta\mapsto \gamma_G(\beta,\beta_0)$ gives a bijective correspondence between 
	\[
	\{ TF,S^V\}^{G}_{\mathcal{U}}(\cong[TF,S^V]^{G}_q)
	\]
	and
	\[
	H^k_{G}(TF,TF^{S^1};\pi_kS^V)
	\]
	by \thmref{corresp}, where $\gamma_{G}(\beta,\beta_0)$ is the $G$-equivariant difference obstruction class.
	\item[(F2)] $H^k(U,\partial U; \Z)$ is isomorphic to $\Z$ if $U$ is orientable, and $\Z_2$ if $U$ is nonorientable.
	\item[(F3)] The forgetful map
	\[
	\bar\varphi\colon H^k_{\Z_2}(U,\partial U; \pi_kS^V)\to H^k(U,\partial U; \Z)\cong\Z\text{ or } \Z_2
	\]
	is given by multiplication of $2$ (Proposition \ref{prop:forget}).
	In particular, 
	$
	\bar{\varphi}(c) \equiv  0 \mod 2
	$
	for 
	\[c\in H^k_{\Pin(2)}(TF,TF^{S^1};\pi_kS^V) \cong H^k_{\Z_2}(U,\partial U; \pi_kS^V).
	\]
\end{itemize}

%(cf. \cite[Proposition 3.3]{BF}, \cite[Theorem 4.1]{BauerRefined}):
Let $\beta_0$ be the unit of the stable cohomotopy group   $\{ TF,S^V\}^{S^1}_{\mathcal{U}}$. 
By (F1),  the correspondence $\beta\mapsto\gamma_{S^1}(\beta,\beta_0)$ for $\beta\in \{ TF,S^V\}^{S^1}_{\mathcal{U}}$ gives a bijective correspondence between $\{ TF,S^V\}^{S^1}_{\mathcal{U}}$  and $H^k_{S^1}(TF,TF^{S^1};\pi_kS^V)$ which is identified with $H^k(U,\partial U; \Z)\cong\Z\text{ or } \Z_2$.
Define $\deg\beta \in \Z_{2}$ by
\[
\deg\beta = \gamma_{S^1}(\beta,\beta_0)\mod 2.
\]

Since the degree $\deg\beta$ is the mod $2$ difference obstruction class of maps essentially from a manifold to a sphere with same dimension, this can be interpreted as a kind of mod 2 mapping degree.
We may assume $\beta$ is represented by a map $g\colon (TF/{S^1}, TF^{S^1})\to (S^V,S^{V^{S^1}})$ which is smooth on the complement of $TF^{S^1}$. 
By perturbing $g$, the image of $g|_{TF^{S^1}}$ is contained in a subspace of $S^{V^{S^1}}$ of codimension $b_+(M)-\dim B\geq 2$. 
A choice of a generic point $v\in S^{V^{S^1}}\setminus \left(\Ima g|_{TF^{S^1}}\right)$ makes the preimage $g^{-1}(v)$  a compact $0$-manifold.
Then  $\deg \beta$ is  the number modulo $2$ of elements in $g^{-1}(v)$.

\begin{Remark}
	For a single $4$-manifold, the moduli space is always orientable and the $\Z$-valued Seiberg--Witten invariants can be defined.
	On the other hand,  the moduli space for a family of $4$-manifolds may be nonorientable.
	Therefore only the $\Z_2$-valued invariants can be defined in general.
\end{Remark}

\section{Spin families}
\label{section: Spin families}

Let $M$ be a closed oriented smooth $4$-manifold with a spin structure $\fs$, $B$ be a closed manifold, and $\pi\colon X\to B$ be a smooth family with fiber $M$.
Let $T(X/B)$ be the tangent bundle along the fibers.
In this subsection we shall discuss when $T(X/B)$ admits a global spin structure modelled on $\fs$ defined in \subsecref{subsection: Families Seiberg--Witten invariants}.

Let $\Diff(M,[\fs])$ be the group of orientation-preserving self-diffeomorphisms $f\colon M\to M$ for which the pulled-back spin structures $f^*\fs$ are isomorphic to $\fs$. 
The notation $[\fs]$ indicates the isomorphism class of $\fs$.
When a diffeomorphism $f$ belongs to $\Diff(M,[\fs])$, we just say that $f$ preserves $\fs$ to avoid complication, although to be precise it should be said that $f$ preserves $[\fs]$.
Let $\Aut(M,\fs)$ be the group of pairs $(f,\hat{f})$, where  $f\in\Diff(M,[\fs])$ and $\hat{f}$ is an automorphism of $\fs$ covering $f$.
Then we have an exact sequence
\begin{align}
\label{eq: exact seq gauge auto diff}
1\to\G(\fs)\to\Aut(M,\fs)\to\Diff(M,[\fs])\to 1,
\end{align}
where $\G(\fs)$ is the gauge transformation group of the spin structure $\fs$, i.e., the group of automorphisms of $\fs$ covering $\id_M$.
Note that $\G(\fs)\cong \{\pm 1\}$.
Taking the classifying spaces, we obtain a fibration
\[
B\G(\fs)\to B\Aut(M,\fs)\to B\Diff(M,[\fs]).
\]
The homotopy class of a map $\tilde{\rho}\colon B\to B\Aut(M,\fs)$ corresponds to
the isomorphism class of a family $\pi\colon X\to B$ with a global spin structure on $T(X/B)$ modelled on $\fs$.
Suppose that a map $\rho\colon B\to B\Diff(M,[\fs])$ is given.
Since $B\G(\fs)=B\Z_2\cong \RP^\infty$, there exists a sole obstruction in $H^2(B;\Z/2)$ to lift $\rho$ to a map $\tilde{\rho}\colon B\to B\Aut(M,\fs)$.
Denote by $O(\rho)$ the obstruction class.

Suppose that we have finitely many
commuting orientation-preserving self-diffeomorphisms $f_1,\ldots,f_n \in \Diff(M,[\fs])$. 
Then we can form the multiple mapping torus
\[
X=X_{f_1,\ldots,f_n} \to T^n.
\]
Let us denote also by $O(f_1,\ldots,f_n)$ the obstruction $O(\rho)$ for such a family, where $\rho\colon T^n\to B\Diff(M,[\fs])$ is the classifying map.

\begin{Proposition}
\label{Proposition: obst van comm lifts}
The obstruction class $O(f_1,\ldots,f_n) \in H^2(T^n;\Z/2)$ is zero if and only if 
 $f_1,\ldots,f_n$
admit lifts
 $\hat{f}_1,\ldots, \hat{f}_n$
  to the spin structure $\fs$ which mutually commute.
\end{Proposition}
\begin{proof}
If the lifts commute, then we can obviously construct a global spin structure by patching a product spin structure on $M\times [0,1]^n$ by the lifts.	Therefore $O(f_1,\ldots,f_n) =0$.

Conversely, suppose that $O(f_1,\ldots,f_n) =0$.
Let $C$ be the CW complex with one $0$-cell and one $1$-cell which forms a circle.
Let $C_1,\ldots,C_n$ be copies of $C$.
A cell structure of $T^n$ is given by the product of $C_1,\ldots,C_n$.
Then 
\begin{itemize}
	\item[$\bullet$] the $1$-skeleton of $T^{n}$ is the wedge sum of  $C_1,\ldots,C_n$, and
	\item[$\bullet$] there is a bijection between the set of $2$-cells and  the set of pairs $(i,j)$ with $i\neq j$.
	Here each pair $(i,j)$ corresponds to a unique $2$-cell $D_{ij}$ bounded by the wedge sum $C_i\vee C_j$. 
	%for every pair $(i,j)$, there is a unique $2$-cell $D_{ij}$ bounded by the wedge sum $C_i\vee C_j$. 
\end{itemize} 
A choice of lifts $\hat{f}_i$ for $f_i$ determines a global spin structure on $\pi^{-1}(\text{$1$-skeleton})$ by identifying the spin structures on the endpoints of the $1$-cells via $\hat{f}_i$.
The class $O(f_1,\ldots,f_n)$ is the obstruction to extend such a spin structure on $\pi^{-1}(\text{$1$-skeleton})$ to $\pi^{-1}(\text{$2$-skeleton})$.
To extend the spin structure on $\pi^{-1}(C_i\vee C_j)$ to $\pi^{-1}(D_{ij})$, the monodromy $\hat{f}_i \hat{f}_j \hat{f}_i^{-1} \hat{f}_j^{-1}$ should be $1 \in \Aut(M,\fs)$.
This completes the proof.
\end{proof}

\begin{Definition}
Let $(M,\fs)$ be a spin $4$-manifold.
Finitely many self-diffeomorphisms $f_1,\ldots,f_n$ on $M$ are called \textit{spin commuting} when 
\begin{enumerate}
\item $f_1,\ldots,f_n$  commute with each other, 
\item each $f_i$ preserves the orientation of $M$ and the spin structure $\fs$,  \item $O(f_1,\ldots,f_n)=0$, or equivalently, there exist commuting lifts $\hat{f}_1\ldots\hat{f}_n$.
\end{enumerate}
If spin commuting diffeomorphisms $f_1,\ldots,f_n$ are given, then we can form a multiple mapping torus $X=X_{f_1,\ldots,f_n}$ with fibers $M$ which admits a global spin structure $\fs_{X}$ modelled on $\fs$. 
We call $(X,\fs_{X})$ the \textit{spin mapping torus associated with the spin commuting diffeomorphisms $f_1,\ldots,f_n$}.  
\end{Definition}

Let us give a sufficient condition for vanishing of $O(f_{1}, \ldots,f_{n})$.
For a self-diffeomorphism $f$ on $M$, the support of $f$ is defined by
\[
\supp f = \overline{\{x\in M\,|\, f(x)\neq x\}},
\]
and so is the support of an element of $\operatorname{Aut}(M,\mathfrak{s})$ similarly.

\begin{Lemma}
\label{lem: obstruction vanishes if the supports are disjoint}
Let $(M,\mathfrak{s})$ be a spin manifold, and $f_{1}, \ldots, f_{n}$ be commuting diffeomorphisms on $M$ preserving the orientation of $M$ and $\fs$.
If $\supp{f}_{1}, \ldots, \supp{f}_{n}$ are mutually disjoint and $M \setminus \supp{f}_{i}$ is connected for each $i$, then $O(f_{1}, \ldots,f_{n})=0$ holds.
\end{Lemma}

\begin{proof}
Because of Proposition~\ref{Proposition: obst van comm lifts},
it suffices to show that there exist commuting lifts $\hat{f}_{1} ,\ldots, \hat{f}_{n}$ of $f_{1}, \ldots, f_{n}$ to the spin structure.
Recall that we have an exact sequence \eqref{eq: exact seq gauge auto diff}.
Let $\tilde{f}_{1}, \ldots, \tilde{f}_{n} \in \operatorname{Aut}(M,\mathfrak{s})$ be lifts of $f_{1}, \ldots, f_{n}$.
For each $i \in \{1, \ldots, n\}$, the lift $\tilde{f}_{i}$ lives in the (spin) gauge group outside $\supp{f}_{i}$:
\[
\tilde{f}_{i}|_{M \setminus \supp{f}_{i}} \in \mathcal{G}(\mathfrak{s}|_{M \setminus \supp{f}_{i}}) = \mathrm{Map}(M \setminus \supp{f}_{i}, \Z/2) \cong \Z/2 = \{\pm1\}.
\]
Define a lift $\hat{f}_{i} \in \operatorname{Aut}(M,\mathfrak{s})$ of $f_{i}$ by
\begin{align*}
\hat{f}_{i} = 
\left\{\begin{array}{cc}
\tilde{f_{i}} \quad\quad {\rm if}\quad \tilde{f}_{i}|_{M \setminus \supp{f}_{i}} = 1,  \\
-1 \cdot \tilde{f}_{i} \quad{\rm if}\quad \tilde{f}_{i}|_{M \setminus \supp{f}_{i}} = -1.
\end{array}\right.
\end{align*}
Then $\hat{f}_{1} ,\ldots, \hat{f}_{n}$ have disjoint supports each other, and hence
mutually commute.
\end{proof}

We note that, on the other hand, the obstruction class $O(f_1,\ldots,f_n)$ may be non-trivial for some example of $f_1,\ldots,f_n$.

\begin{Example}
An example of a multiple mapping torus $X\to T^n$ with nonzero $O(f_1,\ldots,f_n)$ is given as follows.
Let $M$ be $T^3$ equipped with the spin structure $\fs_0$ with trivial $\Spin(3)$-bundle. 
The two-to-one homomorphism $h\colon \Spin(3)\to\SO(3)$ is given by the action of $\Spin(3)=\SP(1)$ on $\Ima \HH$ by conjugation.
Then the multiplication of 
\[
h(i) = \begin{pmatrix}
	1 & 0 &0\\
	0 & -1 & 0\\
	0& 0& -1
\end{pmatrix}\text{ and }
h(j) = \begin{pmatrix}
	-1 & 0 & 0\\
	0& 1 & 0\\
	0&0&-1
\end{pmatrix}
\]
on $\R^3$ induces a pair $(f_1,f_2)$ of commuting diffeomorphisms on $T^3$.
Their lifts $\hat{f}_1,\hat{f}_2$ to $\fs_0$ anticommute since $ij=-ji$. 
Therefore $O(f_1,f_2)$ is not zero.
\end{Example}

So far  we have considered the case of diffeomorphisms, however we can also consider its topological analogue.
Namely, we can consider {\it topological spin structures} as discussed in \cite{MR2644908}
and  also discuss an obstruction $O(f_{1}, \ldots, f_{n})$ to the lifting problem to topological spin structures
 for given commuting {\it homeomorphisms} $f_{1}, \ldots, f_{n}$.
By a parallel argument,  we  have a topological version of Lemma~\ref{lem: obstruction vanishes if the supports are disjoint}:

\begin{Lemma}
\label{lem: obstruction vanishes if the supports are disjoint top version}
Let $M$ be an oriented topological manifold, $\fs$ be a topological spin structure on $M$, and $f_{1}, \ldots, f_{n}$ be commuting homeomorphisms on $M$ preserving the orientation of $M$ and $\fs$.
If $\supp{f}_{1}, \ldots, \supp{f}_{n}$ are mutually disjoint, then $O(f_{1}, \ldots,f_{n})=0$ holds.
\end{Lemma}

\section{Main results}
\label{section: Main Theorem}

In this section, we shall give the main results in this paper and its consequences.
Theorem~\ref{thm:main} is a rigidity theorem for the families Seiberg--Witten invariants on  families of spin $4$-manifolds.
Combining Theorem~\ref{thm:main} with Proposition~\ref{prop:M0}, a non-vanishing of the families Seiberg--Witten invariants shown in \cite{BK} for specific families, we obtain a non-vanishing result for more general families as Corollary~\ref{cor:cor1}.
Theorem~\ref{thm:cor2} gives  a constraint on $b_{+}$ of the fibers of families of spin $4$-manifolds satisfying certain conditions.
This is a family analogue of Furuta's 10/8-inequality~\cite{MR1839478}.

\begin{Theorem}\label{thm:main}
Let $B$ be a closed smooth manifold.
Let $M_1$ and $M_2$ be oriented closed smooth $4$-manifolds with spin structures $\fs_1$ and $\fs_2$ respectively satisfying the following:
\begin{itemize}
	\item[$\bullet$] $b_1(M_1)=b_1(M_2)=0$, $b_+(M_1)=b_+(M_2)\geq \dim B + 2$.
	\item[$\bullet$] $-\dfrac{\sign(M_i)}{4} -1-b_+(M_i)+\dim B =0$ $(i=1,2)$.
\end{itemize}	
For $i=1,2$, let $X_i\to B $ be a smooth fiber bundle with fibers $M_i$ equipped with  a global spin structures $\fs_{X_i}$ modelled on $\fs_i$,
$\ind D_i$ be the virtual index bundle of the family of Dirac operators for $\fs_{X_i}$ and
$H^+_i\to B $ be the bundle of $H^+(M_i)$ associated to $X_{i}$.
Set $\alpha_i= [\ind D_i] - [H_i^+]\in KO_{\Pin(2)}(B)$.
If $\alpha_1=\alpha_2$, then we have
\[
FSW^{\Z_2}(X_1,\fs_{X_1}) = FSW^{\Z_2}(X_2,\fs_{X_2}).
\] 
\end{Theorem}

We shall use the following 
lemma to prove \thmref{thm:main}, which
is fundamental in this section:

\begin{Lemma}\label{lem:forget}
In the setting of Subsection~\ref{subsection: Families Seiberg--Witten invariants}, suppose the condition \eqref{eq:dim}.
Consider the maps
\[
\begin{CD}
 \{TF,S^V\}^{\Pin(2)}_{\mathcal{U}^\prime}@>{\varphi}>>\{ TF,S^V\}^{S^1}_{\mathcal{U}}@>{\deg}>>\Z_2,
 \end{CD}
\]
where $\varphi$ is the forgetful map restricting $\Pin(2)$-action to $S^1$-action.
Then the image of $(\deg\circ\varphi)$ is $\{0\}$ or $\{1\}$, and is determined by $\alpha=[F]-[\underline{V}]\in KO_{\Pin(2)}(B)$.
\end{Lemma}

\begin{proof}
The class $\alpha$ determines $\{TF,S^V\}^{\Pin(2)}_{\mathcal{U}^\prime}$, $\{ TF,S^V\}^{S^1}_{\mathcal{U}}$, $\varphi$, and hence also the image of $(\deg\circ\varphi)$.
For $\beta_1^\prime,\beta_2^\prime\in \{TF,S^V\}^{\Pin(2)}_{\mathcal{U}^\prime}$, the additivity of difference obstruction classes implies that
\[
\deg\varphi(\beta_1^\prime)-\deg\varphi(\beta_2^\prime) = \gamma_{S^1}(\varphi(\beta_1^\prime),\beta_0)-\gamma_{S^1}(\varphi(\beta_2^\prime),\beta_0) =\gamma_{S^1}(\varphi(\beta_1^\prime),\varphi(\beta_2^\prime)).  
\]
%by \eqref{eq:diff}.
Moreover, the fact (F3) in Subsection~\ref{subsection: Families Seiberg--Witten invariants} implies that
\[
\gamma_{S^1}(\varphi(\beta_1^\prime),\varphi(\beta_2^\prime))=\bar{\varphi}(\gamma_{\Pin(2)}(\beta_1^\prime,\beta_2^\prime)) \underset{(2)}{\equiv}0.
\]
Therefore the image of $(\deg\circ\varphi)$ is $\{0\}$ or $\{1\}$.
\end{proof}

\begin{proof}[Proof of \thmref{thm:main}]
We shall use notation in Subsection~\ref{subsection: Families Seiberg--Witten invariants}.
As explained there, it follows that a family of monopole maps for each $(X_i,\fs_{X_i})$ defines a class $\beta_i^\prime$ in $\{TF,S^V\}^{\Pin(2)}_{\mathcal{U}^\prime}$, where $F$ and $V$ satisfy that 
\[
[F] - [\underline{V}] = \alpha_1 =\alpha_2.
\]	
Since the families Sieberg--Witten invariants for $(X_i,\fs_{X_i})$ are given by  
\[
FSW^{\Z_2}(X_i,\fs_{X_i}) =  \deg\varphi(\beta_i^\prime),
\]
\lemref{lem:forget} implies that $FSW^{\Z_2}(X_1,\fs_{X_1})=FSW^{\Z_2}(X_2,\fs_{X_2})$.
\end{proof}
%%%%%%%%%%%%%%%%%%%% Edited by Nakamura 2020/2/10  up to here 

\begin{Remark}
The idea  to use that $\varphi$ is given as multiplication of $2$
has appeared in \cite{MR2259057,MR2264722,B}.
\end{Remark}
\begin{Remark}\label{rem:rep-decomposition}
Let $G=\Pin(2)$.
Since the $G$-action on $B$ is trivial, we have an isomorphism 
\[
KO_G(B) \cong (KO(B)\otimes R(G;\R)) \oplus (K(B)\otimes R(G;\C)) \oplus (KSp(B)\otimes R(G;\HH)),
\]
where
 $R(G;\mathbb{F})$ is the free abelian group generated by  irreducible $G$-representations over the field $\mathbb{F}$
 \cite{MR0248877}.
In our case,
 $[\ind D]$ is in the component $KSp(B)\otimes R(G;\HH)$, and $[H^+]$ is in $KO(B)\otimes R(G;\R)$.
Furthermore, we may assume
\[
[\ind D] = [\ind D]_0\otimes h_1\quad [H^+] = [H^+]_0\otimes \tilde{\R},
\] 
where
\begin{itemize}

\item[$\bullet$]
 $[\ind D]_0\in KSp(B)$ is the class of the index bundle of $D$ 
which is regarded as a non-equivariant $\HH$-linear operator, 

\item[$\bullet$]
$h_1\in R(G;\HH)$ is a representation given by the multiplication of $G$ on $\HH$, 

\item[$\bullet$]
$[H^+]_0$ is the class of $H^+$ in $KO(B)$ 
and $\tilde{\R}$ is the $G$-representation given by composition of 
 the projection $G\to G/S^1=\deux$ with  multiplication on $\R$. 	
 \end{itemize}
\end{Remark}

We shall exhibit an example of families with non-zero families Seiberg--Witten invariants.
Let $M_0= K3 \# n(S^2\times S^2)$ and
$\fs_0$ be the spin structure on $M_0$ which is unique up to isomorphism.
We construct spin commuting diffeomorphisms $f_1,\ldots,f_n$ on $M_0$ as follows.
Let $\varrho$ be an orientation-preserving self-diffeomorphism of $S^2\times S^2$ satisfying the following properties:
\begin{enumerate}
\item There is a $4$-ball $B_0$ embedded in $S^2\times S^2$ such that the restriction of $\varrho$ on a neighborhood $N(B_0)$ of $B_0$ is the identity map on $N(B_0)$.
\item $\varrho$ reverses orientation of $H^+(S^2\times S^2)$.
\end{enumerate}
One way to get such $\varrho$ is  as follows.
Let $\varrho'$ be the orientation-preserving self-diffeomorphism on $S^{2} \times S^{2}$ given by the direct product of the complex conjugation on $S^{2}=\mathbb{CP}^{1}$.
This diffeomorphism $\varrho'$ acts on the intersection form as $(-1)$-multiplication, hence reverses orientation of $H^{+}$.
Obviously $\varrho'$ admits a fixed point,
and deforming $\varrho'$ by isotopy near a fixed point, we can get a fixed ball rather than a fixed point.
Then the deformed diffeomorphism  $\varrho$ satisfies the desired properties.

Choose $n$ disjoint $4$-balls $B_1,\ldots,B_n\subset K3$.
We assume $M_0$ is constructed by removing $B_1,\ldots,B_n$ from $K3$  and gluing $n$ copies of $S^2\times S^2\setminus B_0$.
The construction of $f_i$ is as follows.
Consider $M_0$ as the connected sum of the summand of the $i$-th $S^2\times S^2$ with the remaining part $M_{(i)}:= K3\#(n-1)S^2\times S^2$.
Define $f_i$ by 
\[
f_i = (\varrho\text{ on the $i$-th $S^2\times S^2$})\#\id_{M_{(i)}}.
\]
Note that $f_{1}, \ldots, f_{n}$ obviously commute with each other.
Note that $f_{i}$ preserves orientation of $M$.

\begin{Remark}
\label{Remark: H_{0}}
Let $H^+_0\to T^n$ be the bundle of $H^+(M_0)$.
Let $\ell$ be the unique nontrivial real line bundle over $S^1$ and $\pi_i\colon T^n=S^1\times\cdots\times S^1\to S^1$ be the projection to the $i$-th $S^1$.
Let 
\[
\xi_n=\pi_1^*\ell\oplus\cdots\oplus \pi_n^*\ell.
\]
Then $H^+_0\cong \xi_n\oplus\underline{\R}^3$.
\end{Remark}

The following calculation gives us instances of families with 
non-zero families Seiberg--Witten invariants.

\begin{Proposition}[(\cite{BK})]\label{prop:M0}
For $(M_0,\fs_0)$ and $f_1,\ldots,f_n$ as above, we have that:
\begin{enumerate}
	\item The set $\{f_1,\ldots,f_n\}$ is spin commuting. Let $(X_0,\fs_{X_0})$ be the associated spin mapping torus.
	\item $[\ind D_0]=[\underline{\HH}]\in KO_{\Pin(2)}(T^n)$, where $\ind D_0$ is the Dirac index bundle of $(X_0,\fs_{X_0})$.
	\item $FSW^{\Z_2}(X_0,\fs_{X_0}) =1 \in \Z_{2} = \{0,1\}$.
\end{enumerate}	
\end{Proposition}
\begin{proof}
\lemref{lem: obstruction vanishes if the supports are disjoint} implies the assertion (1).
The assertion (2) will be proved by Lemma~\ref{Lemma: get rid of assum on ind} below. %the index theorem and the sum formula of the indices. 
%(See the proof of Lemma~\ref{Lemma: get rid of assum on ind}.)
To prove the assertion (3), we use  Theorem~1.1 of \cite{BK}.
Let $N=n (S^2\times S^2)$, and assume $M_0=K3\# N$.
Let $H^+_N$ be the bundle of $H^+(N)$. 
Then $H^+_N\cong \xi_n$.
By Theorem~1.1 of \cite{BK},
\[
FSW^{\Z_2}(X_0,\fs_{X_0})=SW(K3,\fs_0|_{K3})\cdot\langle w_n(\xi_n), [T^n]\rangle = 1
\]
holds.
\end{proof}

Combining Remark~\ref{Remark: H_{0}}, \thmref{thm:main} and \propref{prop:M0},
we obtain   the following:

\begin{Corollary}\label{cor:cor1}
Let $M$ be a closed smooth spin $4$-manifold  such that we have a ring isomorphism $H^*(M;\Q)\cong H^*(M_0;\Q)$.
Suppose that we have a smooth fiber bundle $X \to T^{n}$ with fiber $M$ and with a global spin structure $\fs_{X}$ modelled on the given spin structure on $M$.
Let $\ind D$ be the Dirac index bundle of $(X,\fs_X)$ and $H^+\to T^n$ be the bundle of $H^+(M)$ associated to $X$.
Suppose that  $[\ind D]=[\underline{\HH}]\in KO_{\Pin(2)}(T^n)$ and $H^+\cong \xi_n\oplus\underline{\R}^3 $. 
Then we have
\[
FSW^{\Z_2}(X,\fs_X) = 1.
\]	
\end{Corollary}

\begin{Remark}
Morgan--Szab\'{o} \cite{MS} proves the rigidity theorem that every homotopy $K3$ surface admits
 a $\Spinc$-structure with trivial determinant line bundle whose Seiberg--Witten invariant is congruent to $1$ modulo $2$. 
\corref{cor:cor1} can be considered as a family version of Morgan--Szab\'{o}'s theorem. 	
\end{Remark}

The following theorem gives a family version of the 10/8-inequality for families with fiber having $\sign=-16$ and $b_1=0$.

\begin{Theorem}
\label{thm:cor2}
Let $M$ be a spin $4$-manifold with $\sign(M)=-16$ and $b_1(M)=0$.
Suppose that we have a smooth fiber bundle $X \to T^{n}$ with fiber $M$ and with a global spin structure $\fs_{X}$ modelled on the given spin structure on $M$.
Let $\ind D$ and $H^+$ be as in \corref{cor:cor1}.
Suppose $[\ind D]=[\underline{\HH}]$ and %$H^+$ is stably equivalent to $H^+_0$, that is, 
there exists a nonnegative integer $a$ such that
\[
H^+ \cong \xi_n\oplus \underline{\R}^a.
\]	
Then 
\[
b_+(M)\geq n+3
\]
holds.
\end{Theorem}

\begin{proof}
The proof is parallel to an argument in Proposition~2 of \cite{FKM}.
First, it follows from the assumption on $H^{+}$ that $b_+(M)\geq n$.
Suppose $n\leq b_+(M)< n+3$.
Then we have $0\leq a < 3$.
Let $\ind D_0$ and $H_0^+$ be as in \propref{prop:M0} and \remref{Remark: H_{0}}.
Choose a vector bundle $\xi^\prime$ over $T^n$ so that 
\[
-[\xi_n] = [\xi^\prime] - [\underline{\R}^l]\text{ in } KO_{S^1}(T^n).
\]
Then, for some nonnegative integers $x$, $y$, we have
\begin{align*}
[\ind D] -[H^+] &= [\underline{\HH}^{x+1}\oplus \underline{\R}^y\oplus\xi^\prime] - [\underline{\HH}^{x}\oplus \underline{\R}^{y+l+a}],\\
[\ind D_0] -[H^+_0] &= [\underline{\HH}^{x+1}\oplus \underline{\R}^y\oplus\xi^\prime] - [\underline{\HH}^{x}\oplus \underline{\R}^{y+l+3}].	\end{align*}
Let us consider the following commutative diagram:
\[
\xymatrix@C=0pt{ 
& \{TF,S^{V\oplus\tilde{\R}^{3-a}}\}_{\mathcal{U}^\prime}^{\Pin(2)} \ar[rd]^{\varphi_{0}}\\
\{TF,S^{V}\}_{\mathcal{U}^\prime}^{\Pin(2)}\ar[ru]^{\iota_0}\ar[rd]_{\varphi_{1}}& &\{TF,S^{V\oplus\R^{3-a}}\}_{\mathcal{U}}^{S^{1}}&\overset{\deg}{\longrightarrow}&\Z_{2}.\\
& \{TF,S^{V}\}_{\mathcal{U}}^{S^{1}} \ar[ru]_{\iota_1}}
\]
Here $\iota_{i}$ are induced from the inclusion $V \hookrightarrow V\oplus \R^{3-a}$, and $\varphi_{i}$ are the forgetful maps restricting the $\Pin(2)$-actions to the $S^{1}$-actions.
We shall compare two compositions $\deg \circ \varphi_{0} \circ \iota_{0}$ and 
$\deg \circ \iota_{1} \circ \varphi_{1}$ in the diagram.

\propref{prop:M0} and \lemref{lem:forget} imply that the image of the composition $\deg\circ\varphi_0$ is $\{1\}$ and therefore the image of $\deg\circ\varphi_0\circ\iota_0$ is also $\{1\}$.

On the other hand, $S^V$ is $S^{1}$-equivariantly contractible in $S^{V\oplus \R^{3-a}}$, since we assumed that $a<3$ and the $S^{1}$-action on $S^{\R^{3-a}}$ is trivial.
Therefore the image of the composition $\deg \circ \iota_{1} \circ \phi_{1}$ should be $\{0\}$.
This contradicts the commutativity of the diagram.
\end{proof}

In Lemma~\ref{Lemma: get rid of assum on ind}, shown in Section~\ref{section: Calculation of the index bundle}, we shall give a way to replace the assumption that $[\ind{D}]=[\underline{\HH}]$ in Corollary~\ref{cor:cor1} and Theorem~\ref{thm:cor2} with a more geometric condition.
Let us combine Lemma~\ref{Lemma: get rid of assum on ind} with Corollary~\ref{cor:cor1} and Theorem~\ref{thm:cor2} here.

\begin{Corollary}
\label{cor: cor1 revised}
Let $(M,\fs)$ be an oriented closed smooth spin $4$-manifold  with $H^*(M;\Q)\cong H^*(M_0;\Q)$ and $f_1,\ldots,f_n$ be diffeomorphisms on $M$.
Suppose that each of $f_{i}$ preserves $\fs$ and that $\supp{f}_{1}, \ldots, \supp{f}_{n}$ are mutually disjoint.
Then, by Lemma~\ref{lem: obstruction vanishes if the supports are disjoint}, there exist lifts of $f_{1}, \ldots, f_{n}$ to the spin structure.
Fix such lifts and form the spin mapping torus $(X,\fs_X)$.
Let $H^+\to T^n$ be the bundle of $H^+(M)$ for $X$.
Suppose that $H^+\cong \xi_n\oplus\underline{\R}^3$.
Then we have
\[
FSW^{\Z_2}(X,\fs_X) = 1.
\]
\end{Corollary}

\begin{proof}
By Lemma~\ref{Lemma: get rid of assum on ind}, either 
$\ind{D}$ or $-\ind{D}$ is represented by a trivial bundle, and $\ind_{\C}{D}=-\sign(M)/8=2$ by the index theorem.
Therefore the assertion of the corollary follows from Corollary~\ref{cor:cor1}.
\end{proof}

\begin{Corollary}
\label{cor: cor2 revised}
Let $(M,\fs)$ be an oriented closed spin smooth $4$-manifold with $\sign(M)=-16$ and $b_1(M)=0$ and $f_1,\ldots,f_n$ be diffeomorphisms on $M$.
Suppose that each of $f_1,\ldots,f_n$ preserves $\fs$ and that $\supp{f}_{1}, \ldots, \supp{f}_{n}$ are mutually disjoint.
Let $H^+\to T^n$ be the bundle of $H^+(M)$ associated with $f_{1}, \ldots, f_{n}$.
Suppose that there exist a nonnegative integer $a$ such that
$H^+ \cong \xi_n\oplus \underline{\R}^a$.
Then we have
\[
b_+(M)\geq n+3.
\]
\end{Corollary}

\begin{proof}
This follows from Lemma~\ref{lem: obstruction vanishes if the supports are disjoint}, Theorem~\ref{thm:cor2}, Lemma~\ref{Lemma: get rid of assum on ind} and the index theorem as well as the proof of Corollary~\ref{cor: cor1 revised}.
\end{proof}

\section{Applications}
\label{section: Applications}

In this section, we shall give two topological applications of our main results in the previous section.
The first application is to detect non-smoothable actions on $4$-manifolds.
The second is to detect non-smoothable families.
We note that most $4$-manifolds $M$ appearing in this section have non-zero signature, and 
 for such $M$,
we have  $\Diff^{+}(M)=\Diff(M)$ and $\Homeo^{+}(M) = \Homeo(M)$.

Let us denote by $-E_{8}$ the (unique) closed simply connected oriented topological 4-manifold whose intersection form is the negative definite $E_{8}$-lattice \cite[Theorem~1.7]{MR679066}.
In Theorem~\ref{theo: application to nonsmoothable action},
for any $m \geq 3$,
we construct non-smoothable $\Z^{m-2}$-actions
on the topological $4$-manifold $2(-E_{8}) \# m S^{2}\times S^{2}$.
Notice that the $4$-manifold $2(-E_{8}) \# m S^{2}\times S^{2}$ is homeomorphic to $K3\#(m-3)S^{2} \times S^{2}$  and hence  admits a smooth structure.

\begin{Theorem}
\label{theo: application to nonsmoothable action}
Let $m \geq 3$.
Then the topological  (but smoothable) $4$-manifold $M$ defined by
\[
M = 2(-E_{8}) \# m S^{2}\times S^{2}
\]
admits commuting self-homeomorphisms $f_{1}, \ldots, f_{m}$ with the following properties:
\begin{itemize}
\item[$\bullet$] For any distinct numbers $i_{1},\ldots, i_{m-3} \in \{1, \ldots, m\}$, there exists a smooth structure on $M$ for which $f_{i_{1}}, \ldots, f_{i_{m-3}}$ are diffeomorphisms.
\item[$\bullet$] For any distinct numbers $i_{1},\ldots, i_{m-2} \in \{1, \ldots, m\}$, 
there exists no smooth structure on $M$ for which all of 
$f_{i_{1}}, \ldots, f_{i_{m-2}}$ are diffeomorphisms.
\end{itemize}
\end{Theorem}

\begin{proof}
Let us write the connected sum components of $mS^{2} \times S^{2}$ as
\[
mS^{2} \times S^{2} = \#_{i=1}^{m} (S^{2} \times S^{2}) = \#_{i=1}^{m} N_{i}.
\]
For each $i \in \{1, \ldots,m\}$, let
\[
f_{i} : N_{i} \to N_{i}
\]
be an orientation-preserving self-diffeomorphism given by a copy on $N_{i}$ of $\varrho : S^{2} \times S^{2} \to S^{2} \times S^{2}$ given in Section~\ref{section: Main Theorem}.
Since $f_{i}$ has a fixed ball, we can extend $f_{i}$ as a self-homeomorphism onto $M$ by the identity map outside $N_{i}$.
Let us write $f_{i} : M \to M$ also for the extended self-homeomorphism.
Note that obviously $\supp{f}_{1}, \ldots, \supp{f}_{m}$ are mutually disjoint.

We first show that, for any distinct numbers $i_{1},\ldots, i_{m-3} \in \{1, \ldots, m\}$, there exists a smooth structure on $M$ such that $f_{i_{1}}, \ldots, f_{i_{m-3}}$ are diffeomorphisms with respect to the smooth structure.
For simplicity of notation, let us consider the case that $i_{1}=1, \ldots, i_{m-3}=m-3$.
First, note that Freedman's theorem (see for example \cite{MR1201584}) implies that there exists a homeomorphism 
\[
\varphi : 2(-E_{8}) \#_{i=m-2}^{m}N_{i} \to K3.
\]
For $j \in \{1, \ldots, m-3\}$, denote by $B_{j}$ the (topologically) embedded $4$-ball in $2(-E_{8}) \#_{i=m-2}^{m}N_{i}$ which was used to define the extension of $f_{j} : N_{j} \to N_{j}$ onto $M$.
Let $B'_{1}, \ldots, B'_{m-3} \subset K3$ be smoothly embedded disjoint $4$-balls.
%%%%%%%%%%%%% revised by Nakamura, 4/20
%Let $\psi$ be a self-homeomorphism on $2(-E_{8}) \#_{i=m-2}^{m}N_{i}$ which maps $B_{j}$ to $\varphi^{-1}(B'_{j})$ for all $j \in \{1, \ldots, m-3\}$.
We can construct a self-homeomorphism $\psi$ on $2(-E_{8}) \#_{i=m-2}^{m}N_{i}$ which maps $B_{j}$ to $\varphi^{-1}(B'_{j})$ for all $j \in \{1, \ldots, m-3\}$:
by taking a suitable isotopy, we may assume each $\varphi^{-1}(B'_{j})$ is contained in the interior of $B_{j}$. 
Since the boundary spheres of $B_{j}$ and $\varphi^{-1}(B'_{j})$ are locally-flatly embedded, the annulus theorem \cite{MR679069} implies that $\overline{B_{j}\setminus\varphi^{-1}(B'_{j})}$ is homeomorphic to $S^3\times I$.
Then  we can find an ambient isotopy which moves $B_{j}$ to $\varphi^{-1}(B'_{j})$.
%%%%%%%%%%%%%
We can extend the homeomorphism
\[
\varphi \circ \psi : 2(-E_{8}) \#_{i=m-2}^{m}N_{i} \to K3
\]
to a homeomorphism
\[
\phi : M \to K3\#(m-3)S^{2} \times S^{2}
\]
by forming the connected sum along $B_{1}, \ldots, B_{m-3}$ and $B'_{1}, \ldots, B'_{m-3}$ with the identity map on the $(m-3)$-copies of $S^{2} \times S^{2}$.
By construction, the composition
\[
\phi \circ f_{j} \circ \phi^{-1} : K3\#(m-3)S^{2} \times S^{2} \to K3\#(m-3)S^{2} \times S^{2}.
\]
is obviously a diffeomorphism for any $j \in \{1, \ldots, m-3\}$.
This means that $f_{j}$ is a diffeomorphism on $M$ equipped with the smooth structure of $K3\#(m-3)S^{2} \times S^{2}$ via $\phi$.

The rest task is to show that $f_{i_{1}}, \ldots, f_{i_{m-2}}$ are not smoothable at the same time for distinct numbers $i_{1},\ldots, i_{m-2} \in \{1, \ldots, m\}$.
Set $n=m-2$.
Assume that $f_{i_{1}}, \ldots, f_{i_{n}}$ are diffeomorphisms for some smooth structure on $M$.
Let $H^+\to T^n$ be the bundle of $H^+(M)$ associated with $f_{i_{1}}, \ldots, f_{i_{n}}$.
For each $k \in \{1,\ldots,n\}$, the diffeomorphism $f_{i_{k}}$ reverses the orientation of $H^{+}$ for the $i_{k}$-th component of $S^{2} \times S^{2}$ of
$2(-E_{8}) \# m S^{2}\times S^{2}$ and $f_{i_{k}}$ acts trivially on $H^{+}$ for the rest connected sum component.
Thus we have $H^{+} \cong \xi_{n} \oplus \underline{\R}^{2}$, and therefore we can apply Corollary~\ref{cor: cor2 revised} to $f_{i_{1}}, \ldots, f_{i_{n}}$.
It follows from this corollary that $b_{+}(M) \geq n+3 = m+1$, but obviously $b_{+}(M) = m$.
This is a contradiction.

This completes the proof of Theorem~\ref{theo: application to nonsmoothable action}. 
\end{proof}

\begin{Remark}
\label{rem: precious works on nonsmoothable actions}
Non-smoothable actions have been studied by many authors, but for groups having several generators, there is only little previous work.
Here we explain such work briefly and compare it with Theorem~\ref{theo: application to nonsmoothable action}.
The third author~\cite{MR2644908} constructed a non-smoothable $\Z^{2}$-action on the connected sum of an Enriques surface and $S^{2} \times S^{2}$.
Y. Kato~\cite{Kato} constructed a non-smoothable $(\Z/2)^{2}$-actions on certain spin $4$-manifolds with $|\sign| \geq 64$.
D.~Baraglia~\cite{Baraglia} constructed $\Z^{2}$-actions and $(\Z/2)^{2}$-actions on certain non-spin $4$-manifolds.
In these results, each of generators of $\Z^{2}$ or $(\Z/2)^{2}$ can be realized as a smooth diffeomorphism for some smooth structure, so they are similar to Theorem~\ref{theo: application to nonsmoothable action} in this sense.
However, on the other hand, Theorem~\ref{theo: application to nonsmoothable action} provides a non-smootable $\Z^{n}$-action for all $n \geq 2$ and the $4$-manifold acted by $\Z^{n}$ is different from that in all of the work explained in this Remark.
\end{Remark}

Let $M$ be an oriented topological (but smoothable) manifold, 
$B$ be a smooth manifold, and $M \to X \to B$ be a fiber bundle whose structure group is $\operatorname{Homeo}(M)$.

We say that the bundle $X$ is {\it smoothable as a family} or $X$ has {\it a smooth reduction}, if there exists a smooth structure on $M$ such that there is a reduction of the structure group of $X$ to $\operatorname{Diff}(M)$ with respect to the smooth structure via the inclusion $\operatorname{Diff}(M) \hookrightarrow \operatorname{Homeo}(M)$.
If $X$ is not smoothable as a family, we say that $X$ is {\it non-smoothable as a family} or $X$ has {\it no smooth reduction}.

\begin{Remark}
\label{rem: extention of non-smoothable family}
For a topological fiber bundle  $X \to B \times B'$,
if the restriction $X|_{B} \to B$ is non-smoothable as a family, then so is $X\to B \times B'$.
\end{Remark}

In Theorem~\ref{theo: application to nonsmoothable family}, 
we shall construct a non-smoothable family whose fiber is the topological $4$-manifold
$
2(-E_{8}) \# m S^{2}\times S^{2}$.
Here we use the following notation.
Set $[m]=\{1,2,\ldots,m\}$.
For the $m$-torus $T^{m}$ and a subset $I=\{i_{1},\ldots,i_{k}\} \subset [m]$ with cardinality $k$,
denote by 
$
T^{k}_{I}
$
the embedded $k$-torus in $T^{m}$ defined as the product of the $i_{1}, \ldots, i_{k}$-th $S^{1}$-components.

\begin{Theorem} \label{theo: application to nonsmoothable family}
	Let $3\leq m \leq 6$.
	Let $M$ be the topological (but smoothable) $4$-manifold defined by
	\[
	M = 2(-E_{8}) \# m S^{2}\times S^{2}.
	\]
	Then there exists a $\operatorname{Homeo}(M)$-bundle \[
	M \to X \to T^{m}
\]
over the $m$-torus with the following properties: 
	let $I=\{i_{1},\ldots,i_{k}\} \subset [m]$ be a subset with cardinality $k$.
	\begin{itemize}
		\item[$\bullet$] The total space $X$ admits a smooth manifold structure.
		\item[$\bullet$] If $k\leq m-3$,  
		the restricted family
		\[
		X|_{T^{k}_{I}} \to T^{k}_{I}
		\]
		has a reduction to $\Diff(M)$ for some smooth structure on $M$. 
		\item[$\bullet$] If $m-2\leq k\leq m$, 
		the restricted family, 
		\[
		X|_{T^{k}_{I}} \to T^{k}_{I}
		\]
		has no reduction to $\Diff(M)$ for any smooth structure on $M$.
	\end{itemize}
\end{Theorem}

\begin{proof}
Let $f_{1}, \ldots, f_{m}$ be the commuting self-homeomorphisms on $M$ constructed in the proof of Theorem~\ref{theo: application to nonsmoothable action}.
Let $M \to X \to T^{m}$ be the multiple mapping torus for $f_{1}, \ldots, f_{m}$.
Then $X$ is a $\Homeo(M)$-bundle.
Note that, because of Lemma~\ref{lem: obstruction vanishes if the supports are disjoint top version},
there exists a global topological spin structure on the bundle $X$.

First, smoothability of $X$ as a manifold will be verified 
in Proposition~\ref{prop: smoothing of the total spaces}
in Section~\ref{section: Smoothing of the total spaces}.

Second, we shall verify by contradiction that $X|_{T^{k}_{I}} \to T^{k}_{I}$ has no reduction to $\Diff(M)$ for any smooth structure on $M$ if $I=\{i_{1},\ldots, i_{k}\}$ with  $k\geq m-2$.
We shall show the non-smoothability for $m-2\leq k\leq 4$, but this is enough also for general $k \geq m-2$ by Remark~\ref{rem: extention of non-smoothable family}.
Assume that $X|_{T^{k}_{I}}$ could be  smoothable as a family for some smooth structure on $M$.
Then the global topological structure induces a global spin structure and 
 we have the family of Dirac operators $\ind{D}$ associated with $X|_{T^{k}_{I}}$.
We shall show \lemref{lem: dim of base is less than 3} and 
Lemma~\ref{lem: Novikov theorem type argument}  in Section~\ref{section: Calculation of the index bundle}, and they ensure triviality  $\ind{D}=[\underline{\HH}]$.
%(In the case that $m=4,5$, this also follows from Lemma~\ref{lem: dim of base is less than 3}.)
Moreover, the bundle of $H^{+}$ associated with $X|_{T^{k}_{I}}$ satisfies that 
$H^{+} \cong \xi_{k} \oplus \underline{\R}^{a}$ for $a=m-k$, as explained in the proof of Theorem~\ref{theo: application to nonsmoothable action}.
Therefore we can apply Theorem~\ref{thm:cor2} to $X|_{T^{k}_{I}} \to T^{k}_{I}$, and the inequality $b_{+}(M) \geq k+3\geq m+1$ should hold, but obviously $b_{+}(M) = m$.
This is a contradiction.

Lastly let us  check that $X|_{T^{k}_{I}} \to T^{k}_{I}$ is smoothable as a family for $I=\{i_{1},\ldots, i_{k}\}$ with $k\leq m-3$.
The restriction  $X|_{T^{k}_{I}}$ is the multiple mapping torus of $f_{i_{1}}, \ldots, f_{i_{k}}$.
By Theorem~\ref{theo: application to nonsmoothable action},
there exists a smooth structure on $M$ such that $f_{i_{1}}, \ldots, f_{i_{k}}$ are diffeomorphisms.
Therefore the structure group of $X|_{T^{k}_{I}}$ obviously reduces to $\operatorname{Diff}(M)$ with respect to this smooth structure.
\end{proof}

\begin{Remark}
The assertion on non-smoothability of $f_{i_{1}}, \ldots, f_{i_{m-2}}$ in Theorem~\ref{theo: application to nonsmoothable action} obviously follows from Theorem~\ref{theo: application to nonsmoothable family}.
\end{Remark}

\begin{Remark}
In the case of $m=3$ in Theorem~\ref{theo: application to nonsmoothable family}, the second condition on the smoothability of $X|_{T^{k}_{I}}$ turns no condition.
\end{Remark}

\begin{Remark}
The non-smoothability of $X$ given in Theorem~\ref{theo: application to nonsmoothable family} in the case that $m=3$ follows from Morgan--Szab\'{o}~\cite{MS} without using Theorem~\ref{theo: application to nonsmoothable family} as follows.
The family $X$ in the case that $m=3$ is a $\operatorname{Homeo}(M)$-bundle $M \to X \to T^{k}$ with $k \in \{1,2,3\}$ and $M = 2(-E_{8}) \# 3S^{2} \times S^{2}$.
This bundle is given as the multiple mapping torus for commuting homeomorphisms supported in the $3S^{2}\times S^{2}$-components.
Assume that the family $X$ is smoothable as a family.
Let us take a smooth structure on $M=2(-E_{8}) \# 3S^{2} \times S^{2}$ for which the structure group of $X$ has a reduction to the diffeomorphism group.
Consider the unique spin structure on the smooth $4$-manifold.
This $4$-manifold has non-zero Seiberg--Witten invariant for the spin structure by \cite{MS}, and from this we can deduce that there does not exist a diffeomorphism which reverses the orientation of $H^{+}$.
By restricting the family to $S^{1}$ embedded into $T^{k}=(S^{1})^{k}$ as the first factor, we can get a smoothable family over the circle $M \to X|_{S^{1}} \to S^{1}$.
Since this restricted family is the mapping torus of the homeomorphism $f_{1}$, the smoothability of $X|_{S^{1}}$ implies that $f_{1}$ is topologically isotopic to a diffeomorphism $g$ on $M$.
Since $f_{1}$ reverses the orientation of $H^{+}(M)$, so does $g$.
This is a contradiction.
\end{Remark}

One  can verify  a bit stronger result  on the smoothability of $X|_{S^{1}}$ for any $S^{1}$ embedded in $T^{m}$ in Theorem~\ref{theo: application to nonsmoothable family}.

\begin{Proposition}
\label{prop: application to nonsmoothable family 2}
Let $4\leq m\leq 6$ and let $M \to X \to T^{m}$ be the $\operatorname{Homeo}(M)$-bundle given in Theorem~\ref{theo: application to nonsmoothable family}.
Then for any homeomorphism $\varphi : M \to K3\#(m-3)S^{2} \times S^{2}$ and any embedding of $S^{1}$ to $T^{m}$,
the structure group of $X|_{S^{1}}$ reduces to $\Diff(M)$, where $\Diff(M)$ is the diffeomorphism group with respect to the smooth structure on $M$ defined as that of $K3\#(m-3)S^{2} \times S^{2}$ via $\varphi$. 
\end{Proposition}

\begin{proof}
Equip $M$ with a smooth structure through $\varphi$.
Take an embedding of $S^{1}$ into $T^{m}$.
Note that $X|_{S^{1}}$ can be regarded as the mapping torus of a homeomorphism $g$ on $M$.
Recall the following two classical results:
\begin{itemize}

\item[$\bullet$]
 Every algebraic automorphism of the intersection form of $M \cong K3 \# (m-3)S^{2} \times S^{2}$
  is induced from a diffeomorphism by a result of Wall~\cite{Wall}.
  
  \item[$\bullet$]
An algebraic automorphism of the intersection form corresponds to a topological isotopy class by a result of Quinn~\cite{MR868975}.
\end{itemize}
Therefore there exists a diffeomorphism on $M$ which is topologically isotopic to $g$.
This means that the structure group of $X|_{S^{1}}$ reduces to $\Diff(M)$.
\end{proof}

Let us denote by $\Homeo(M) \sslash \Diff(M)$
the  {\it homotopy quotient}:
\[
\Homeo(M) \sslash \Diff(M) := \left(E\Diff(M) \times \Homeo(M)\right)/\Diff(M).
\]

\begin{Corollary}
\label{cor: the fun group of quotient}
We have 
\[
\pi_{1}(\Homeo(K3\# S^{2} \times S^{2}) \sslash \Diff(K3\# S^{2} \times S^{2})) \neq 0.
\]
\end{Corollary}

\begin{proof}
Set $M=K3\# S^{2} \times S^{2}$.
The case that $m=4$ of Theorem~\ref{theo: application to nonsmoothable family} and Proposition~\ref{prop: application to nonsmoothable family 2}  implies that the fundamental group of the homotopy fiber of the natural map $B\Diff(M) \to B\Homeo(M)$ is non-trivial.
To finish the proof, just recall that this homotopy fiber is homotopy equivalent to $\Homeo(M) \sslash \Diff(M)$.
\end{proof}

\begin{Remark}
Note that the argument of the  proof in  Corollary~\ref{cor: the fun group of quotient} is valid also for the 
$4$-manifold as $Z\#S^{2} \times S^{2}$ instead of $K3\#S^{2} \times S^{2}$, where $Z$ is an exotic $K3$.
However, we do not know an example of $Z$ such that $Z\#S^{2} \times S^{2}$ is not diffeomorphic to $K3\#S^{2} \times S^{2}$.
\end{Remark}

\section{Calculation of the index bundle}
\label{section: Calculation of the index bundle}

In this section, we shall give a few ways to give a sufficient condition for the Dirac index bundle associated with a given family of $4$-manifolds to be trivial.
The results given in this section have been used in the previous sections.

\begin{Lemma}
\label{Lemma: get rid of assum on ind}
Let $M$ be a closed spin $4$-manifold.
Let $f_1,\ldots,f_n$ be spin commuting diffeomorphisms on $M$.
If $\supp{f}_{1}, \ldots, \supp{f}_{n}$ are mutually disjoint,
then either the spin Dirac index bundle $\ind D$ associated with $f_1,\ldots,f_n$ or $-\ind{D}$ is represented by a trivial bundle.
\end{Lemma}

\begin{proof}
We shall use the excision formula of the index of families of Fredholm operators, and for the sake of it, decompose $M$ into $n$-pieces of codimension-0-submanifolds
\[
M = M_{1} \cup_{Y_{1}} \cdots \cup_{Y_{n-1}} M_{n}
\]
so that $\supp{f_{i}} \subset M_{i}$ for each $i$ as follows.
Set $N_{0} := M$.
Let us define closed subsets $A_{1}, B_{1} \subset N_{0}$ by $A_{1} := \supp{f}_{1}$ and $B_{1} := \supp{f}_{2} \sqcup \cdots \sqcup \supp{f}_{n}$.
By Urysohn's lemma, we can take a continuous function $\tilde{\chi}_{1} : N_{0} \to [-1,1]$ such that $\tilde{\chi}_{1}(A_{1}) = \{-1\}$ and $\tilde{\chi}_{1}(B_{1})=\{1\}$.
By perturbing $\tilde{\chi}_{1}$, we can get a smooth function $\chi_{1} : N_{0} \to [-3/2, 3/2]$ such that $\chi_{1}(A_{1}) \subset [-3/2,-1]$ and $\chi_{1}(B_{1}) \subset [1,3/2]$.
By Sard's theorem, for a generic point $\epsilon \in (-1,1)$, the inverse image $Y_{1} := \chi_{1}^{-1}(\epsilon)$ is a $3$-dimensional closed submanifold of $N_{0}$.
Define $M_{1} := \chi_{1}^{-1}([-3/2,\epsilon])$ and $N_{1} := \chi_{1}^{-1}([\epsilon,3/2])$.
Then we get a decomposition into codimension-$0$-submanifolds
 of $N_{0} = M_{1} \cup_{Y_{1}} N_{1}$ along $Y_{1}$.
Next, let us define closed subsets $A_{2}, B_{2} \subset N_{1}$ by $A_{2} := \supp{f}_{2} \sqcup Y_{1}$ and $B_{2} := \supp{f}_{3} \sqcup \cdots \sqcup \supp{f}_{n}$.
By the same procedure, we can get a decomposition of $N_{1}$ into codimension-$0$-submanifolds along a $3$-dimensional submanifold $Y_{2}$ of $\mathop{\rm{int}}{N_{2}}$:
 $N_{1} = M_{2} \cup_{Y_{2}} N_{2}$.
 Note that $Y_{1}$ is a closed $3$-manifold.
Proceeding these steps inductively, we can get a decomposition of $M$ into codimension-$0$-submanifolds
\[
M = M_{1} \cup_{Y_{1}} \cdots \cup_{Y_{n-1}} M_{n}
\]
along closed $3$-manifolds $Y_{1}, \ldots, Y_{n-1}$.
By construction, each $\supp{f}_{i}$ is contained in $M_{i}$.
Let $M_{i} \to X_{i} \to S^{1}$ be the mapping cylinder of $f_{i}$.
This $X_{i}$ is a bundle of a smooth $4$-manifold with boundary.
Our multiple mapping cylinder $M \to X \to T^{n}$ is regarded as the fiberwise sum of $\pi_{1}^{\ast}X_{1}, \ldots, \pi_{n}^{\ast}X_{n}$ along trivial bundles $Y_{1} \times T^{n} \to T^{n}, \ldots, Y_{n-1} \times T^{n} \to T^{n}$, where $\pi_{i} : T^{n} \to S^{1}$ is the $i$-th projection.
Denote by $\hat{M}_{i}$ the cylindrical $4$-manifold obtained by gluing $M_{i}$ with $\partial M_{i} \times [0, \infty)$.
Then we can get a bundle of a cylindrical $4$-manifold $\hat{M}_{i} \to \hat{X}_{i} \to S^{1}$, and can define 
the family of spin Dirac operators $D_{i}$ on $\hat{X}_{i}$.
Then, under suitable weighted Sobolev norms (for example, see Donaldson's book~\cite[Subsubsection~3.3.1]{MR1883043}), we
can obtain
\begin{align}
\label{eq: excision formula of the index of families}
[\ind{D}] = [\ind{\pi_{1}^{\ast} D_{1}}] + \cdots + [\ind{\pi_{n}^{\ast} D_{n}}]
\end{align}
in $KO_{\Pin(2)}(T^{n})$ by the excision formula of the index of families.

Since $[\ind D_i]\in KSp(S^1)\otimes R(G;\HH)$ and $KSp(S^1)=KSp(pt)=\Z$, $\ind{D_{i}}$ or $-\ind{D_{i}}$ is represented by a trivial quaternion bundle in $KO_{\Pin(2)}(S^{1})$ (see Remark \ref{rem:rep-decomposition}).
Hence  $\ind{D}$ or $-\ind{D}$ is the same  by \eqref{eq: excision formula of the index of families}.
\end{proof}

\begin{Remark}
	Note that we cannot  apply  Lemma~\ref{Lemma: get rid of assum on ind} for the proof of Theorem~\ref{theo: application to nonsmoothable family}.
	%The reason why we cannot use Lemma~\ref{Lemma: get rid of assum on ind} for the family $X$ given in the proof of Theorem~\ref{theo: application to nonsmoothable family} is as follows.
	To verify   non-smoothability of $X|_{T^k_I}$ as a family, 
	 we argue by contradiction, and for this
	 we assume that $X|_{T^k_I}$ has a reduction to the diffeomorphism group with respect to some smooth structure of the fiber.
	However this assumption does not guarantee that $f_{i_1}, \ldots, f_{i_k}$ are diffeomorphisms, but just homeomorphisms.
	%	Therefore we cannot apply Lemma~\ref{Lemma: get rid of assum on ind} for $X$.
%	In the proof of Lemma~\ref{Lemma: get rid of assum on ind}, the condition that each $f_{i}$ is a diffeomorphism is crucially used.
%	In fact, the condition that each $f_{i}$ is a diffeomorphism ensures that each $X_{i}$ is a smooth family,
%	and if we cannot ensure that $X_{i}$ is a smooth family, we cannot define $\ind{D_{i}}$,
%	then we cannot use the additivity of the indices.
\end{Remark}

\begin{Remark}
	One can deduce the assumption $\ind{D}=[\underline{\HH}]$  in  Corollary~\ref{cor:cor1} and Theorem~\ref{thm:cor2} 
	from the following stronger but more geometric condition, 
	which is  different  from Lemma~\ref{Lemma: get rid of assum on ind}.	
		 Assume that there exists  a Riemannian metric on $M$ which is invariant under the pull-backs of all $f_{1}, \ldots, f_{n}$.
	For example, this assumption is satisfied if  all of $f_{1}, \ldots, f_{n}$ have finite order. Indeed,
	 the group generated by them  is  finite since 
	$f_{1}, \ldots, f_{n}$  mutually commute, and then
	we can obtain an invariant metric 
	by taking the average of any metric by the action of this finite  group.
	Let us  derive $\ind{D}=[\underline{\HH}]$ assuming the existence of an invariant metric
	 $g_{0}$ for $f_{1}, \ldots, f_{n}$.
	The metric $g_{0}$ gives a ``constant'' fiberwise metric on the mapping torus for $f_{1}, \ldots, f_{n}$.
	Note that we can employ this fiberwise metric in the process of a finite dimensional approximation of a family of Seiberg--Witten equations described in Subsection~\ref{subsection: Families Seiberg--Witten invariants}, since any genericity for the metric is not necessary for the finite dimensional approximation.
		Then the index bundle $\ind{D}$ is clearly trivial.
	Because of the usual index calculation, the complex rank of the fiber of $\ind{D}$ is $|{\rm sign}(M)|/8=2$.
	Therefore we obtain $\ind{D}=[\underline{\HH}]$.
\end{Remark}

The index bundle is always trivial when  
 the base space is a low-dimensional torus:

\begin{Lemma}
	\label{lem: dim of base is less than 3}
	Let $(M,\fs)$ be a closed spin $4$-manifold.
	Let $B$ be a closed manifold and
	 $X\to T^k$ be a fiber bundle with fibers $M$ with a global spin structure $\fs_X$ modelled on $\fs$.
	Let $[\ind D] \in KO_{\Pin(2)}(T^k)$ denote the class of the (virtual) index bundle of the family of spin Dirac operators associated to $X$.
	If $k \leq 3$, then $[\ind D]$ or $-[\ind D]$ is represented by a trivial quaternionic vector bundle.
\end{Lemma}

Before proving Lemma~\ref{lem: dim of base is less than 3}, we need some preliminaries.
By \remref{rem:rep-decomposition}, we may assume that $[\ind D] $ is in $KSp(T^n)\otimes R(\Pin(2);\HH)$ and can be written as
\[
[\ind D] = [\ind D]_0\otimes h_1,
\] 
where $[\ind D]_0\in KSp(T^n)$ is the class of the index bundle of $D$ as an non-equivariant $\HH$-linear operator, and $h_1\in R(\Pin(2);\HH)$ is the representation given by the multiplication of $\Pin(2)$ on $\HH$.
Then we have the following useful decomposition of the $KSp$-groups of $T^n$.

\begin{Proposition}[({\cite[Lemma 31 and Remark 32]{FK}})]\label{prop:FK-decomposition}
	For integers $q$ and $p$ with $p\geq 0$, we have an isomorphism
	\[
	KSp^q(T^n\times\R^p)\cong \bigoplus_{S\subset[n]}KSp^q(\R^S\times\R^p),
	\]
	where $S$ runs through all the subsets of $[n]=\{1,2,\ldots,n\}$ and $\R^S$ is defined as follows: let $\R_k$ be 
	 the $k$-th component of $\R^n$.  Then $\R^S=\prod_{k\in S}\R_k$ if $S\neq\emptyset$, 
	 and $\R^S=\{pt\}$ if $S=\emptyset$.
\end{Proposition}
\begin{proof}
	Consider the exact sequence
	\begin{equation}\label{eq:KSp-exact}
	\begin{CD}
	\cdots  @>>>KSp^q((T^n,T^{n-1})\times\R^p)@>{j^*}>>KSp^q(T^n\times\R^p)@>{h^*}>>KSp^q(T^{n-1}\times\R^p)@>>>  \cdots .
	\end{CD}
	\end{equation}
	By using excision, the first term is identified with
	\[
	KSp^q((T^n,T^{n-1})\times (\R^p\cup\{\infty\},\{\infty\})) \cong KSp^q(T^{n-1}\times\R^{p+1}).
	\]
	Then $j^*$ is identified with the push-forward map $i_{!}\colon KSp^q(T^{n-1}\times \R^{p+1})\to KSp^q(T^n\times\R^p)$ induced from an open embedding $i\colon \R\to T^1\subset T^n$.
	Let $\pi\colon T^n\times\R^p\to T^{n-1}\times\R^p$ be the projection. 
	Then $\pi^*$ gives a right-inverse of $h^*$. 
	Therefore the above sequence splits. Moreover
	$h^*$ is a surjection, and then 
	$j^*$ turns out to be an injection.
	 Thus we obtain an isomorphism
	\[
	i_!+\pi^*\colon KSp^q(T^{n-1}\times\R^{p+1})\oplus KSp^q(T^{n-1}\times\R^p)\overset{\cong}{\to} KSp^q(T^n\times\R^p).
	\] 
	By an induction on the cardinality $|S|$ of $S$, the proposition is proved.
\end{proof}

\begin{proof}[Proof of \lemref{lem: dim of base is less than 3}]
	Note that $KSp(pt)=\Z$, $KSp(\R^q)=0$ for $q=1,2,3$.
	(See e.g. \cite[Chapter 11]{Switzer}.)
	If $k\leq 3$,  \propref{prop:FK-decomposition} implies that
	\[
	KSp(T^k) \cong KSp(pt)\cong \Z.
	\]
	This means that every element in $KSp(T^k)$ is represented by a trivial bundle and classified by its rank over $\HH$ if $k\leq 3$.
	Therefore $[\ind D]_0$, and hence $[\ind D]$, is represented by a trivial bundle.
\end{proof}

The main part of this section is devoted to prove the following lemma.
The argument is based on the celebrated result by Novikov 
 that the rational Pontrjagin classes are topological invariants.

\begin{Lemma}
\label{lem: Novikov theorem type argument}
Let $M\to X \to T^{m}$ be the topological bundle given in the proof of Theorem~\ref{theo: application to nonsmoothable family}.
For $I\subset[m]$ with $k=|I|=4$, suppose $X|_{T_{I}^4}$ has a smooth reduction. 
Then the Dirac index bundle satisfies $[\ind{D}]=[\underline{\HH}]$.
\end{Lemma}

Before giving the proof of \lemref{lem: Novikov theorem type argument}, let us describe a strategy of the proof and give some preliminaries.
Denote $X|_{T_{I}^4}$ by $X_I$ and $T_I^k$ by $T_I$. 
Suppose $X_I$ is smoothable as a family and a smooth reduction is given.
We will proceed in the following way:

\begin{itemize}

\item[$\bullet$]
 We verify that the forgetful map $c\colon KSp(T^4)\to K(T^4)$ is injective. 
 
 \item[$\bullet$]
Hence it suffices to prove that the image of $[\ind D]$ under $c$ is represented by a  trivial complex bundle.

\item[$\bullet$]
Since the complex $K$-group of the base space $T_I$ is torsion-free, it suffices to prove that $Ch(\ind D)$, the image of $\ind D$ under the Chern character, is in $H^0(T_I;\Q)$.

\item[$\bullet$]
By the index theorem for families \eqref{eq:family-ind}, it suffices to check 
 $p_{1}^{2}=0$ and $p_{i}=0$ for $i \geq 2$, where $p_i$ are the rational Pontrjagin classes of the tangent bundle along the fibers $T(X_I/T_I)$ of $X_I\to T_I$.

\end{itemize}

It is well-known that the rational Pontrjagin classes of a $\R^n$-bundle depend only on its topological type. 
In fact, the rational Pontrjagin classes can be defined not only for a vector bundle, but also for a topological $\R^n$-bundle whose structure group is in the  group $TOP_n$ of self-homeomorphisms on $\R^n$ preserving the origin.
Furthermore the rational Pontrjagin classes of a bundle are determined by the isomorphism classes as topological bundles, and do not depend on vector bundle structures on them. 
(See Rudyak~\cite[Chapter 3]{MR3467983}, for example. This generalizes the Novikov's theorem.)
Therefore the rational Pontrjagin classes $p_i$ of the {\it tangent micro-bundle} along the fibers $\tau(X_I/T_I)$ 
 are defined over  the underlying topological $\R^n$-bundle of $T(X_I/T_I)$ without using the smooth structure.
We will prove the required properties of $p_i$ directly for $\tau(X_I/T_I)$ from the construction of the topological bundle $X_I$.

To proceed with the above strategy, we recall some classical objects in differential topology.

\subsection{Universal Pontrjagin classes}
Let us recall the rational Pontrjagin classes for topological $\R^n$ bundles.
(See \cite{MR3467983}, for example.)
It is known that the forgetful map $\alpha\colon BO\to BTOP$ induces an isomorphism of their rational cohomology groups 
\[
\alpha^*\colon H^*(BTOP;\Q)\overset{\cong}{\to} H^*(BO;\Q).
\]
Recall that $H^*(BO;\Q)$ is generated by the universal Pontrjagin classes $p_{i}^{\rm univ}$.
Then we have
\[
H^*(BTOP;\Q) \cong H^*(BO;\Q) \cong \Q[p_{1}^{\rm univ}, p_{2}^{\rm univ},\ldots]
\]
via the identification $\alpha^*$.
The stable class of a topological $\R^n$-bundle $\xi\to B$ is classified by its classifying map $t\colon B\to BTOP$.
Define the $i$-th rational Pontrjagin class $p_i(\xi)$ by
\[
p_i(\xi)=t^*p_{i}^{\rm univ}.
\]

\subsection{Rational localization}
Below  we  utilize the $\Q$-localizations $BO[0]$ and $BTOP[0]$ of $BO$ and $BTOP$.
The existence of these $\Q$-localizations is guaranteed 
by the fact that 
both of $BO$ and $BTOP$ are infinite loop spaces, and hence $H$-spaces.
(See Theorem~A and Theorem~C of \cite{MR0236922}.)
In general an  $H$-space is a simple space, and hence is a nilpotent space 
for which a $\Q$-localization can be constructed. 
(See, for example, Corollary~1.4.5 and \S5.3 of \cite{MR2884233}.)
%It is well-known that any nilpotent space has its rationalization.
%Therefore $BO$ and $BTOP$ have their rationalization.

\subsection{Tangent micro-bundle}
We clarify notion of the {\it tangent micro-bundle along the fibers} of a topological bundle $M\to X\overset{\pi}{\to} B$.
Denote the fiber of $X$ over $b\in B $ by $M_b$ and 
define the space $E$ by
\[
E=\{(x,y)\in X\times X\,|\, y\in M_{\pi(x)}\}.
\]
Note that $E$ contains the diagonal set $\Delta_X=\{(x,x)\,|\,x\in X\}$.
The tangent micro-bundle along the fibers $\tau(X/B)$ of $M\to X{\to} B$ is defined as
\begin{equation}\label{eq:micro}
\begin{CD}
\tau(X/B)\colon X@>{\Delta}>> E @>{\pi_1}>> X,
\end{CD}
\end{equation}
where $\Delta$ is the diagonal map and $\pi_1$ is the projection to the first component.
It is easy to check the following properties of $\tau(X/B)$:
\begin{itemize}
	\item[$\bullet$] The sequence \eqref{eq:micro} defines a micro-bundle in Milnor's sense~\cite{MR0161346}.
	(By Kister--Mazur's theorem~\cite{MR0180986}, the micro-bundle determines a topological $\R^n$-bundle 
	which is unique up to isomorphism.)
	\item[$\bullet$]
	If the structure group of $M\to X\to B$ is reduced to $\Diff(M)$ for some smooth structure on $M$,
	then $\tau(X/B)$ is the underlying micro-bundle of the tangent bundle along the fibers $T(X/B)$.

 \end{itemize}

%The basic tool and idea, given as Lemma~\ref{lem: contractibity of the classifying map}, to prove Lemma~\ref{lem: Novikov theorem type argument} is similar to a proof of Novikov's theorem, the topological invariance of the rational Pontrjagin classes.
%(See Rudyak~\cite{MR3467983}.)
%To state Lemma~\ref{lem: contractibity of the classifying map},
%we need the rationalization $BO \to BO_{0}$ and $BTOP \to BTOP_{0}$ of $BO$ and $BTOP$.
%Note that the existence of these rationalization can be seen as follows.
%First, $BO$ and $BTOP$ are infinite loop spaces, and hence H-spaces.
%(See Theorem~A and Theorem~C of \cite{MR0236922}.)
%In general, any H-space is a simple space, and hence a nilpotent space.
%(See, for example, Corollary~1.4.5 of \cite{MR2884233}.)
%It is well-known that any nilpotent space has its rationalization.
%Therefore $BO$ and $BTOP$ have their rationalization.

\begin{Lemma}
\label{lem: contractibity of the classifying map}
Let $f : N \to N$ be an orientation-preserving diffeomorphism
on $N=S^{2} \times S^{2}$.
Assume that $f$ has a fixed embedded ball $B^{4} \subset N$.
Let $N_{f} \to S^{1}$ be the mapping torus of $f$.
Define a map
\[
\phi : (T^{k-1} \times N_{f}, T^{k} \times B^{4}) \to (BO[0],{\rm pt})
\]
as the composition of the classifying map $T^{k-1} \times N_{f} \to BO$
of the tangent bundle along the fibers $T((T^{k-1} \times N_{f})/T^{k})$ of the fiber bundle $T^{k-1} \times N_{f} \to T^{k}$ and the natural map $BO \to BO[0]$.
Then $\phi$ is homotopic to the constant map onto $({\rm pt}, {\rm pt}) \subset (BO[0],{\rm pt})$.
\end{Lemma}

\begin{proof}
%As a consequence of Kirby--Siebenmann theory, one can see that the natural map $BO_{0} \to BTOP_{0}$ is a homotopy equivalence.
%(See, for example, 3.6.5~Lemma of \cite{MR3467983}.)
%Therefore the all of $H^{\ast}(BO;\Q), H^{\ast}(BO_{0};\Q),H^{\ast}(BTOP;\Q)$ and $H^{\ast}(BTOP_{0};\Q)$ are naturally isomorphic each other.
%Henceforth we identify the all of them via the natural isomorphisms.
%Recall that $H^{\ast}(BO;\Q)$ is the polynomial ring
%$\Q[p_{1}^{\rm univ}, p_{2}^{\rm univ},\ldots]$, where $p_{i}^{\rm univ}$ are the Pontrjagin classes of the universal bundle.
%
Note that 
\[
\pi_i(BO[0]) =\pi_i(BO)\otimes\Q= \left\{
\begin{aligned}
\Q &\quad i = 0\mod 4,\\
0 \,&\quad  i\neq 0\mod 4.
\end{aligned}\right.
\]
Let $u : BO[0] \to BO[0]$ be the identity map and $u_{0} : BO[0] \to \{{\rm pt}\} \subset BO[0]$ be the map onto a point in $BO[0]$.
Then the primary difference obstruction $\delta(u,u_{0})$ is non-zero in $H^{4}(BO[0];\pi_{4}(BO[0])) \cong H^{4}(BO[0];\Q) \cong \Q[p_{1}^{\rm univ}]$.
Therefore there exists non-zero number $r \in \Q \setminus \{0\}$ such that $r\delta(u,u_{0}) = p_{1}^{\rm univ}$.

In general, let $N$ be an oriented closed and simply connected $4$-manifold,
$\tau : N \to BO[0]$ be the composition of the classifying map $N \to BO$ of the tangent bundle of $N$ and the natural map $BO \to BO[0]$, and $\tau_{0} : N \to \{{\rm pt}\} \to BO[0]$ be the map onto a point of $BO[0]$.
We claim that $\tau$ is homotopic to $\tau_0$ if $p_1(N)=0$.
Since $\pi_i(BO[0]) =0$ for $0< i<4$ and  $H^i(N;\Q)=0$ for $i>4$,  the difference obstruction $\delta(\tau,\tau_{0})\in H^4(N;\pi_4(BO[0]))$ is the sole obstruction to homotoping $\tau$ to $\tau_0$.  
Because of the naturality of the obstruction class, we have
$p_{1}(N) = \tau^{\ast} p_{1}^{\rm univ} = r \tau^{\ast} \delta(u,u_{0}) = r \delta(\tau,\tau_{0})$ in $H^{4}(N;\Q)$.
Therefore, if $p_{1}(N)=0$, we have $\delta(\tau,\tau_{0})=0$, and hence $\tau$ is homotopic to the constant map $\tau_0$. 
In particular, if we take $N = S^{2} \times S^{2}$, since $S^{2} \times S^{2}$ has trivial signature, we have $p_{1}(S^{2} \times S^{2})=0$, and thus we can deduce that $\tau$ is homotopic to a constant map onto a point in $BO[0]$.
Similarly, if we fix an embedded ball $B^{4} \subset S^{2} \times S^{2}$ and fix a trivialization of $T(S^{2} \times S^{2})$ over $B^{4}$, we can conclude that the pairwise map $\tau : (S^{2} \times S^{2}, B^{4}) \to (BO[0], {\rm pt} )$ is homotopic to the map onto $({\rm pt}, {\rm pt} ) \subset (BO[0], {\rm pt} )$.

Next, let $f : N \to N$ be an orientation-preserving diffeomorphism
on $N=S^{2} \times S^{2}$.
Assume that $f$ has a fixed embedded ball $B^{4} \subset N$.
Let $N_{f} \to S^{1}$ be the mapping torus of $f$.
By the Serre spectral sequence, one can easily see  that $H^{4}(N_{f};\Q) \cong H^{4}(N;\Q)$,
and $p_{1}(T(N_{f}/S^{1}))$ corresponds to $p_{1}(N)$ via this isomorphism, therefore we have $p_{1}(T(N_{f}/S^{1}))=0$.
Using $T(N_{f}/S^{1})$ instead of $T(N)$ in the last paragraph,
we can see that the composition $\tau_{f} : N_{f} \to BO[0]$ of the classifying map $N_{f} \to BO$ of $T(N_{f}/S^{1})$ and the natural map $BO \to BO[0]$ is homotopic to a constant map onto a point in $BO[0]$.
Similarly, the map $\tau_{f} : (N_{f}, S^{1} \times B^{4}) \to (BO[0],{\rm pt})$ is homotopic to the constant map onto $({\rm pt}, {\rm pt}) \subset (BO[0],{\rm pt})$.
Let $\phi : T^{k-1} \times N_{f} \to BO[0]$ be the composition of the classifying map of $T((T^{k-1} \times N_{f})/T^{k}) \to T^{k-1} \times N_{f}$ and the natural map $BO \to BO[0]$.
Since $\phi = p^{\ast}\tau_{f}$, where $p : T^{k-1} \times N_{f} \to N_{f}$,
we have that $\phi$ is homotopic to a constant map, and similarly $\phi : (T^{k-1} \times N_{f}, T^{k-1} \times S^{1} \times B^{4}) \to (BO[0],{\rm pt})$
is homotopic to the constant map onto $({\rm pt}, {\rm pt}) \subset (BO[0],{\rm pt})$.
\end{proof}

To proceed  the proof of Lemma
\ref{lem: Novikov theorem type argument},
let us describe 
several facts about $K$-theory.
Firstly it is easy to see that the complex $K$-group of $T^n$ admits a direct sum decomposition
\begin{equation}\label{eq:K-decomposition}
K(T^n)\cong \bigoplus_{S\subset[n]}K(\R^S).
\end{equation}
(The proof is parallel to that of \propref{prop:FK-decomposition}.)

Let
 $c\colon KSp(B)\to K(B), c_S\colon KSp(\R^S)\to K(\R^S)$ be 
 the forgetful maps which forget the quaternion structures.
\begin{Lemma}\label{lem:forgetKSp}
The forgetful map $c\colon KSp(T^n)\to K(T^n)$ is identified with the direct sum of the forgetful maps $c_S\colon KSp(\R^S)\to K(\R^S)$:
\[
c=\sum_{S\subset[n]} c_S.
\]
\end{Lemma}
\begin{proof}
The forgetful map $c$ builds a bridge between 
 the exact sequence \eqref{eq:KSp-exact} and the corresponding exact sequence of the complex $K$-groups,
 which gives rise to a 
commutative diagram:
\[
\begin{CD}
KSp(T^{n-1}\times\R)\oplus KSp^q(T^{n-1})@>{i_!+\pi^*}>{\cong}> KSp(T^n)\\
@V{c}VV @V{c}VV\\
K(T^{n-1}\times\R)\oplus K(T^{n-1})@>{i_!+\pi^*}>{\cong}> K(T^n).
\end{CD}
\]
Thus $c\colon KSp(T^n)\to K(T^n)$ is identified with the direct sum of the forgetful maps:
 $$KSp(T^{n-1}\times\R)\oplus KSp(T^{n-1})\to K(T^{n-1}\times\R)\oplus K(T^{n-1})$$ via the isomorphisms $i_!+\pi^*$.
The lemma is proved by inductition.
\end{proof}

\begin{proof}[Proof of Lemma~\ref{lem: Novikov theorem type argument}]
Suppose $X_I$ is smoothable as a family and a smooth reduction is given.
Let $T(X_I/T_I) \to X_I$ be the tangent bundle along the fibers.
By \propref{prop:FK-decomposition} and \eqref{eq:K-decomposition}, we have  the splittings
\begin{align*}
KSp(T^4)\cong & KSp(pt)\oplus KSp(\R^4)\cong \Z\oplus \Z, \\ 
K(T^4)\cong & K(pt)\oplus K(\R^4)\cong \Z\oplus \Z.
\end{align*}
Since $c_S\colon KSp(\R^S)\to K(\R^S)$ is injective if $S=\emptyset$ or $|S|=4$, \lemref{lem:forgetKSp} implies that the forgetful map $c\colon KSp(T^4)\to K(T^4)$ is injective.
Therefore, 
in order to verify 
 $[\ind D]=[\underline{\HH}]$, it suffices to check   that $c([\ind D])=[\underline{\C^2}]$.
	
Since $K(T^k)$ is torsion free, the Chern character  $Ch : K(T^k) \to H^{\rm even}(T^k;\Q)$ is also injective. 
The index theorem for families \cite[Theorem~(5.1)]{MR279833} gives the equality
\begin{equation}\label{eq:family-ind}
Ch(c([\ind D])) = \int_{\rm fiber} \hat{A}(T(X_I/T_I)),
\end{equation}
and the integrand is expressed by a polynomial  of rational  Pontrjagin classes, and 
so belongs to $H^{4\ast}(T_I;\Q)$.
%However since $\dim{B} \leq 3$, the only non-trivial part is $\int_{\rm fiber} -p_{1}(T(X/B))/24 \in H^{0}(B ;\Q)$.
%Through $Ch$, the degree zero part of $H^{\rm even}(B;\Q)$ comes from the equivalence class of a trivial bundle or $(-1)$-multiplication of it.
Denote by $p_{i} = p_{i}(T(X_I/T_I)) \in H^{4i}(X_I;\Q)$ the $i$-th rational Pontrjagin classes of $T(X_I/T_I)$.

Once we have seen the vanishings
 $p_{1}^{2}=0$ and $p_{i}=0$ for $i \geq 2$ in $H^{\ast}(X_I;\Q)$, then the
  $\hat{A}$-genus of $T(X_I/T_I)$ is given by
   $\hat{A}(T(X_I/T_I)) = 1 -p_{1}/24$.
Then $Ch(c([\ind D]))$ is in $H^0(T_I;\Q)=\Q$ and actually it coincides with 
$-\sign(M)/8 =2$.
This implies $c([\ind D])=[\underline{\C^2}]$. 

Therefore it suffices to verify that $p_{1}^{2}=0$ and $p_{i}=0$ for $i \geq 2$.
Note that  $p_{i}=0$ hold  for $i \geq 3$,
 since $\rank_{\R}{T(X_I/T_I)}=4$.
Therefore we just need to check that $p_{1}^{2}=0$ and $p_2=0$.
We shall verify  such vanishings  directly in the topological category  from the construction of $X$ as follows.

Set $M^\prime=2(-E_8)\#(m-k)S^2\times S^2$ and  $W = T^k \times M^\prime$. % be the product of a manifold $B$ and an oriented closed topological $4$-manifold $M$.
Let $\tau M^\prime$ and $\tau(W/T^k)$ be the tangent micro-bundle of $M'$ and the one  along the fibers of $W$ respectively.
Thus we have $\tau(W/T^k) \cong \pi_{2}^{\ast}\tau M^\prime$, where $\pi_{2} : W \to M^\prime$ is the projection.
Therefore it follows from degree reason that
\begin{equation}\label{eq: for trivial e8 bundle}
p_{1}(\tau(W/T^k))^{2}=0,\quad p_{i}(\tau(W/T^k))=0 \text{ for $i \geq 2$}.
\end{equation}

Now decompose  $X_I$ as 
\[
X_I = \left(T^{k} \times (M^\prime \setminus \sqcup_{i=1}^{k} B^{4}_{i}) \right) \cup \bigsqcup_{i=1}^{k} (T^{k-1} \times \left( N_{f_{i}} \setminus S^{1} \times B^{4}_{i}) \right),
\]
where $N=S^{2} \times S^{2}$ and $B^{4}_{i}$ are embedded balls.
Let 
\[
\kappa\colon X_I \to T^{k} \times M^\prime
\]
be the collapsing map which collapses each $T^{k-1} \times \left( N_{f_{i}} \setminus S^{1} \times B^{4}_{i}) \right)$-part into $T^k\times \ast$. 
%Let $\psi : X \to BTOP[0]$ be the composition of the classifying map $X \to BO$ of the tangent microbundle along the fibers $T(X/T^{m})$ and the natural map $BO \to BO[0] \to BTOP[0]$.
Let $\psi : X_I \to BTOP[0]$ be the composition of the classifying map $X_I \to BTOP$ of the tangent micro-bundle along the fibers $\tau(X_I/T_I)$ with the natural map $BTOP \to BTOP[0]$.
Let $\psi^\prime\colon T^{k} \times M^\prime\to BTOP[0]$ be the similarly defined map.
By Lemma~\ref{lem: contractibity of the classifying map}, 
the restriction
\[
\psi : (T^{k-1} \times N_{f_{i}}, T^{k} \times B^{4}_{i}) \to (BTOP[0], {\rm pt})
\]
is homotopic to the constant map onto $({\rm pt},{\rm pt}) \subset (BTOP[0],{\rm pt})$.
Then the following diagram is homotopy-commutative:
\[
\xymatrix{
	X_I\ar[r]^{\psi} \ar[d]_{\kappa} &BTOP[0].\\
	T^{k} \times M^\prime\ar[ur]_{\psi^\prime}
}
\]
Thus we have \[
p_{i} = \psi^{\ast}p_{i}^{\rm univ} = \kappa^*(\psi^\prime)^*p_{i}^{\rm univ}=\kappa^*p_{i}(\tau((T^{k} \times M^\prime)/T^{k}).
\] 
By combining this with \eqref{eq: for trivial e8 bundle},
we obtain $p_{1}^{2}=0$ and $p_{i}=0$ for $i \geq 2$.
This completes the proof of the lemma.
\end{proof}

\begin{Remark}
One  can verify  $c([\ind D])=[\underline{\C^2}]$ in a more general setting.
In fact, the following can be shown by an argument above:
for arbitrary $m$, let $M$ and $f_1,\ldots,f_m$ be as in \thmref{theo: application to nonsmoothable action}. 
Let $M\to X\to T^m$ be the multiple mapping torus for $f_1,\ldots,f_m$.
If $X_I\to T_I$ is smoothed as a family for any $I\subset [m]$,  then
we have $c([\ind D])=[\underline{\C^2}]$.

On the other hand, for the proof of \thmref{theo: application to nonsmoothable family},
we need  $\ind D=[\underline{\HH}]$, but  the forgetful map $c\colon KSp(\R^q)\to K(\R^q)$ is not injective if $q\equiv 5,6$ mod $8$.
This is reason why
the argument of the proof of
\thmref{theo: application to nonsmoothable family} is valid only when $m\leq 6$, $k\leq 4$. 
\end{Remark}

\section{Smoothing of the total spaces}
\label{section: Smoothing of the total spaces}

In this section, we give a proof of the smoothablitiy of the total spaces of the non-smoothable families in 
Theorem~\ref{theo: application to nonsmoothable family}.
A basic tool in this section is Kirby--Siebenmann theory~\cite{MR0645390}.
We refer readers to Rudyak's expository book \cite{MR3467983} or the ``Essays"~\cite{MR0645390}.

\begin{Lemma}
\label{lemma: times circle}
The topological $5$-manifold $S^{1} \times 2(-E_{8})$ admits a smooth structure.
\end{Lemma}

\begin{proof}
For a topological manifold $W$, let us denote by $\Delta(W) \in H^{4}(W;\Z/2)$ the Kirby--Siebenmann invariant.
If $W$ is of dimension $5$ and written as $W = S^{1} \times N$ for a simply-connected and closed topological $4$-manifold $N$,
we have $H^{4}(W;\Z/2) \cong H^{4}(N;\Z/2)$ by the K\"{u}nneth theorem, and $\Delta(W)$ corresponds to $\Delta(N)$ via this isomorphism.
This follows from the definition of the Kirby--Siebenmann invariant as an obstruction class (see, for example, \cite[Subsection~3.4]{MR3467983}).
Since the Kirby--Siebenmann invariant is additive with respect to the connected sum of topological $4$-manifolds, we have $\Delta(2(-E_{8}))=0$, and thus we get $\Delta(S^{1} \times 2(-E_{8}))=0$.
Recall that, for a closed topological manifold of dimension $5$, the Kirby--Siebenmann invariant is the only obstruction to the smoothability.
(This follows from the celebrated theorem $TOP/PL \simeq K(\Z/2,3)$ by Kirby and Siebenmann, stated in \cite[page xii]{MR3467983}, and $\pi_{k}(TOP/PL) = \pi_{k}(PL/DIFF)$ for $k <7$ \cite[p. 318]{MR0645390}.)
Therefore this proves that $S^{1} \times 2(-E_{8})$ is smoothable.
\end{proof}

Following Schultz's survey~\cite{Schultz},
we give a smoothing result of a topological embedding of a circle into a higher-dimensional smooth manifold:

\begin{Lemma}
\label{lemma: smoothing of a topological embedding}
Let $W$ be a smooth manifold of dimension $d \geq 5$,
and $f : S^{1} \times \R^{d-1} \to W$ be a topological embedding, i.e. a homeomorphism onto its image.
Then there exists a topological isotopy 
\[
\{F_{t} : S^{1} \times \R^{d-1} \to f(S^{1} \times \R^{d-1}) \subset W\}_{t \in [0,1]}
\]
such that $F_{0} = f$ holds and $F_{1} : S^{1} \times \R^{d-1} \to W$ is a smooth embedding.
\end{Lemma}

\begin{proof}
Set $U := f(S^{1} \times \R^{d-1})$.
We can equip the open topological manifold $U$ with the smooth structure defined as the restriction of the smooth structure of $W$, and also with the smooth structure coming from the standard smooth structure of $S^{1} \times \R^{d-1}$ via $f$.
By Kirby--Sieebenmann theory (see page 194 of the Essays~\cite{MR0645390}, and note that ``concordant implies isotopy" in $\dim \geq 5$), there is a bijection from the set of smoothing of $U$ up to isotopy to $[U,TOP/O] \cong [S^{1},TOP/O]$, which is just a single point since $TOP/O$ is known to be $2$-connected.
Hence smoothing of $U$ is unique up to isotopy.
Therefore there exists a diffeomorphism
\[
g : S^{1} \times \R^{d-1} \to U,
\]
where $U$ is equipped with the restricted smooth structure of $W$, and a topological isotopy
\[
F_{t} : S^{1} \times \R^{d-1} \to U
\]
such that $F_{0} = f$ and $F_{1}=g$.
\end{proof}

The following proposition is the goal of this section:

\begin{Proposition}
\label{prop: smoothing of the total spaces}
The total spaces $X$ of the non-smoothable families given in 
Theorem~\ref{theo: application to nonsmoothable family} are smoothable as manifolds.
\end{Proposition}

\begin{proof}
Set $W = S^{1} \times 2(-E_{8})$, which admits a smooth structure by Lemma~\ref{lemma: times circle}.
Henceforth we fix a smooth structure on $W$.
Fix a point $p \in 2(-E_{8})$ and whose disk-like neighborhood $B^{4} \subset 2(-E_{8})$.
Then the map $S^{1} \to W$ given by $t \mapsto (t,p)$ induces a topological embedding $f : S^{1} \times \R^{4} \to W$.
Note that $T^{n} \times (2(-E_{8}) \setminus B^{4}) = T^{n-1} \times (W \setminus S^{1} \times B^{4})$, where $S^{1} \times B^{4}$ is the image of $f$.
By Lemma~\ref{lemma: smoothing of a topological embedding},
$f$ can be deformed into a smooth embedding $g : S^{1} \times \R^{4} \to W$ via a topological isotopy.
This gives a homeomorphism 
\[
\varphi : X_{1} \to X_{1}',
\]
where
\begin{align*}
&X_{1} := (T^{n} \times 2(-E_{8})) \setminus (T^{n-1} \times f(S^{1} \times \R^{4})), \\
&X_{1}' := (T^{n} \times 2(-E_{8})) \setminus (T^{n-1} \times g(S^{1} \times \R^{4})).
\end{align*}
Note that, although $X_{1}$ is just a topological manifold, $X_{1}'$ is a smooth manifold.

Let $f_{1}, \ldots, f_{m}$ be the homeomorphisms used in the construction of $X$ in Theorem~\ref{theo: application to nonsmoothable family}.
Recall that they act trivially on $2(-E_{8})$, and smoothly on $mS^{2} \times S^{2}$.
Let $E \to T^{n}$ be the mapping torus of $m S^{2} \times S^{2}$ by commuting diffeomorphism $f_{1}, \ldots, f_{m}$.
Let $D^{4}$ be a fixed ball common for all of $f_{1}, \ldots, f_{m}$.
(If we need, we may find such a ball by deforming $f_{1}, \ldots, f_{m}$ by smooth isotopy.)
Then $X$ can be regarded as a topological manifold obtained by gluing the topological manifold $X_{1}$ and a smooth manifold $X_{2} := E \setminus (T^{n} \times D^{4})$ via a homeomorphism.
Since $X_{1}$ is homeomorphic to a smooth manifold $X_{1}'$ via $\varphi$, the topological manifold $X = X_{1} \cup X_{2}$ is also homeomorphic to a smooth manifold, namely, $X$ is smoothable as a manifold.
\end{proof}

\appendix{}

\section{Equivariant obstruction theory}\label{sec:obstruction}

In this appendix, for readers' convenience, we summarize some basic materials of equivariant obstruction theory.
See tom Dieck's book \cite{tomDieck} for details.
Henceforth we denote by $G$ a compact Lie group.

\subsection{$G$-CW complexes}

A {\it $G$-CW complex} is a CW complex $X$ whose $n$-cells are of the forms $G / H_{\sigma} \times D^n$, where 
$H_{\sigma} \subset G$ are closed subgroups of $G$ and $D^{0}=\{\rm pt\}$.
Here the characteristic map of each cell is assumed to be a $G$-map $G / H_{\sigma} \times S^{n-1} \to X^{n-1}$, where $X^{n-1}$ denotes the $(n-1)$-skeleton of $X$.

For   a pair of $G$-CW complexes $(X,A)$, we always assume that 
$G$ acts on $X \backslash A$ freely.
Consider 
the long exact sequence of homology groups over $\Z$
\[
\cdots  \rightarrow H_{n+1}(X^{n+1},X^n) \xrightarrow{\partial}  H_n(X^n, X^{n-1})  \rightarrow  \cdots.
\]
Let $G_0 \subset G$ be the identity component. 
Then $G/G_0$ acts
on each $H_n(X^n, X^{n-1})$, and
hence
\[C_n(X,A):= H_n(X^n, X^{n-1})\]
is a $\Z[G/G_0]$-module.
Let $M$ be a $\Z[G/G_0]$-module.
Then we have the cochain complex
\[
C_G^*(X,A):= \Hom_{\Z[G/G_0]} (C_*(X,A); M)
\]
whose cohomology group $H_G^*(X,A;M)$ is 
called the {\em Bredon} cohomology.

\begin{Lemma}\label{cochain-iso}
There is a chain isomorphism
\[
C_*(X,A) \cong C_*(X/G_0, A/G_0).
\]
\end{Lemma}

\begin{proof}
Let
$\phi: \sqcup_j \ G \times (D^n_j, S^{n-1}_j) \to (X^n, X^{n-1})$
be the characteristic maps.
By excision, we have the isomorphisms
\[
\oplus_j \ H_n (G \times (D^n_j, S^{n-1}_j)) \cong H_n(X^n, X^{n-1}).
\]
The former is isomorphic to
\[
\oplus_j \ H_n (G/ G_0 \times (D^n_j, S^{n-1}_j)) \cong H_n(X^n/G_0, X^{n-1}/G_0)
\]
by another excision.
\end{proof}

\begin{Corollary}\label{cohom-iso}
We have 
\[
H_G^n(X,A;M) \cong H_{G/G_0}^n(X/G_0,A/G_0;M).
\]
\end{Corollary}

Let $Y$ be a path connected $G$-space.
Assume moreover
  that $Y$ is $n$-simple in the sense that the action of $\pi_1(Y,y_0)$ on $\pi_n(Y, y_0)$ is trivial.
Then we have a one-to-one correspondence between $\pi_n(Y,y_0)$ and $[S^n, Y]$, the space of free homotopy of maps.
The $G$-action on $Y$ induces a homomorphism
$G/G_0 \to \text{Aut}(\pi_n(Y))$,
which gives a $\Z[G/G_0]$-module structure on $\pi_n(Y)$.

\begin{Example}\label{Pin(2)}
Let $G = \Pin(2)$.
Then $G_0 =S^1$ and $G/G_0 = \Z_2$.
Let $V$ be a finite dimensional unitary representation of $G$ with 
$\dim V =n$.
The one point compactification $S^V$ of $V$ naturally admits a $G$-action, and hence
$\Z_2$ acts on 
\[
M := \pi_n(S^V) \cong \Z
\]
through the quotient homomorphism
$G \to G/G_0 = \{ \pm 1\}$.

Let $U$ be a manifold on which $G$  acts freely.
We shall consider the pair $(X,A) =(U, \partial U)$.
Define a bundle $l$ over $U/G$ with fiber $\Z$ by
\[
l := U \times_G \Z \to U/ G.
\]
Then we have isomorphisms
\[
H^n_G(U, \partial U; M) \cong H^n_{\Z_2} (U/G_0, \partial U/ G_0; M) 
\cong
H^n (U/G, \partial U/ G; l).
\] 

\end{Example}

\subsection{$G$-equivariant obstruction class}
Let $X, Y$ be path connected $G$-spaces.
Assume also that $Y$ is $n$-simple.
\begin{Theorem}\label{exact-seq}
For $n \geq 1$, there is an exact sequence
\[
[X^{n+1}, Y]^G \to \text{im} \left( [X^n, Y]^G \to [X^{n-1}, Y]^G \right)
\xrightarrow{C^{n+1}} H^{n+1}_G(X_A; \pi_n Y).
\]
\end{Theorem}
A sketch of the construction of the map $C^{n+1}$ above is as follows.
Fix $[h] \in [ X^n, Y]^G$.
Let us consider the diagram
\[
H_{n+1}(X^{n+1}, X^n) \xleftarrow{\rho} 
\pi_{n+1}(X^{n+1}, X^n) \xrightarrow{\partial} \pi_n(X^n) 
\xrightarrow{h_*} \pi_n (Y) =[S^n, Y],
\]
where $\rho$ is the Hurewicz homomorphism.
Since $Y$ is $n$-simple, we have that
$\ker \rho =  \langle x - \alpha x | \alpha \in \pi_1(X^n)\rangle$ and that
$h_* \circ \partial (x - \alpha x) =0$.
Thus we obtain a well-defined cochain
\begin{align*}
C^{n+1} (h) \in 
 C_G^{n+1}(X,A; \pi_n(Y))  & = \Hom_{\Z[G/G_0]} (C_{n+1}(X,A); \pi_n(Y)) \\
& = 
\Hom_{\Z[G/G_0]} (H_{n+1}(X^{n+1},X^n); \pi_n(Y)) 
\end{align*}
given by
$C^{n+1}(h) : = h_* \circ \partial \circ \rho^{-1}$.
This constrcution gives a map 
\[
C^{n+1} : [X^n, Y]^G \to C_G^{n+1}(X,A; \pi_n(Y)),
\]
which induces $C^{n+1}$ in Theorem~\ref{exact-seq}.

\begin{Proposition}\label{homot-ext}
Let $Y$ be an $(n-1)$-connected and $n$-simple space.
Then an arbitrary continuous map
$f: A \to Y$ is extendable to a continuous map
$\tilde{f} : X^n \to Y$.
Moreover any two of such extensions are homotopic to each other rel $A$.
\end{Proposition}

Let $f: A \to Y$ be a continuous map.
The {\it primary obstruction class} is given by
\[
\gamma(f) := C^{n+1}(\tilde{f}) \in H^{n+1}_G(X,A; \pi_n(Y)),
\]
where $\tilde{f}: X^n \to Y$ is an extension.
Define
\[
(X^*,A^*) := (I, \partial I) \times (X,A) 
= (X \times I, I \times A \cup \partial I \times X).
\]
Given $F: I \times A \cup \partial I \times X \to Y$, 
let us denote 
$f_i = F|_{\{i\} \times X}$ for $i =0,1$.
The {\em difference obstruction class between $f_{0}$ and $f_{1}$} is defined by
\[
\gamma(f_0,f_1) := C^{n+1}(F) \in H^{n+1}_G(X^*, A^*; \pi_n(Y))
\cong H^n_G(X,A; \pi_n(Y)),
\]
where the last isomorphism is the suspension isomorphism.

\begin{Theorem}\label{corresp}
Fix a map $f_\ast : X \to Y$.
Let us denote
\[
[X, Y]^G_A : = 
\left\{ \ G\text{-homotopy classes rel } A \text{ of }
f: X \to Y \text{ with }
f|_A = f_\ast|_A  \ \right\}.
\] 
Then we have a one-to-one correspondence
\[
[X, Y]^G_A \leftrightarrow
H^n_G(X,A; \pi_n(Y))
\]
given by $ f \leftrightarrow \gamma(f, f_\ast)$.
\end{Theorem}

\subsection{The image of forgetful map}
Let $U$ be an $n$-dimensional compact (possibly non-orientable) manifold
with boundary $\partial U \ne \emptyset$.
Assume that $\Z_2$ acts freely on the pair $(U, \partial U)$.
Let $\pi$ be the quotient map
\[
\pi : 
(U, \partial U) \to (\bar{U}, \bar{\partial U}) = (U/ \Z_2, \partial U / \Z_2).
\]
Consider a real $n$-dimensional repesentation $V$ of $\Z_2$.
For $Y = S^V$, $\Z_2$ acts on $\pi_n(Y) \cong \Z$.

\begin{Proposition}[({\it cf.} {\cite[II.4]{tomDieck}})]\label{prop:forget}
	The image of the forgetful map
	\[
	\varphi\colon H_{\Z_2}^n(U, \partial U; \pi_n(Y))\to H^n(U, \partial U; \Z)
	\]
	is $2\Z \subset \Z \cong H^n(U, \partial U; \Z)$ if $U$ is orientable, and is $\{0\}$ if $U$ is non-orientable.
\end{Proposition}
\begin{proof}
Let us consider the bundle
\[
l := U \times_{\Z_2} \pi_n(Y) \to \bar{U}.
\]
Since the $\Z_2$-action on $(U,\partial U)$ is free,  the Bredon cohomology $H_{\Z_2}^n(U, \partial U; \pi_n(Y))$ is identified with the $l$-coefficient cohomology $H^n(\bar{U}, \bar{\partial U} ; l)$, and we have  the commutative diagram
\[
\begin{CD}
H_{\Z_2}^n(U, \partial U; \pi_n(Y))   @>{\cong}>> H^n(\bar{U}, \bar{\partial U} ; l)  \\
@V{\varphi}VV @VV{\pi_*}V \\
H^n(U, \partial U; \Z) @= H^n(U, \partial U; \Z).
\end{CD}
\]
Then the conclusion follows from the commutative diagram below
\[
\begin{CD}
H^n(\bar{U}, \bar{\partial U} ; l)@>>> H^n(\bar{U}, \bar{\partial U}; l \otimes \Z_2) 
@>{\cong}>> \Hom(H_n(\bar{U}, \bar{\partial U};\Z_2), \Z_2)\cong\Z_2 \\
@V{\pi^*}VV @VVV @V{0}VV\\
H^n(U, \partial U ; \Z)@>>> H^n(U, \partial U; \Z_2) @>{\cong}>> \Hom(H_n(U, \partial U;\Z_2), \Z_2)\cong\Z_2. \\
\end{CD}
\]
Here the right vertical map is zero since it is induced from the quotient map of the double cover
$(U, \partial U) \to ( \bar{U}, \bar{\partial U})$.
\end{proof}

\end{document}